\documentclass{article}
\usepackage[utf8]{inputenc}
\usepackage[english]{babel}
\usepackage{blindtext}
\usepackage{amssymb}
\usepackage{amsthm}
\usepackage{amsmath}
\usepackage{tikz-cd}
\usepackage{setspace}
\usepackage{thmtools}
\usepackage{thm-restate}
\usepackage{hyperref}
\usepackage{cleveref}
\usepackage[margin=1.2in]{geometry}

\usepackage{graphicx}
\graphicspath{ {images/} }

\theoremstyle{definition}
\newtheorem{definition}{Definition}[section]

\theoremstyle{plain}
\newtheorem{prop}{Proposition}[section]

\theoremstyle{plain}
\newtheorem{theorem}{Theorem}[section]

\theoremstyle{plain}
\newtheorem{corollary}{Corollary}[section]

\theoremstyle{plain}
\newtheorem{lemma}{Lemma}[section]

\theoremstyle{remark}
\newtheorem{remark}{Remark}[section]

\theoremstyle{remark}

\theoremstyle{plain}
\newtheorem{conjecture}{Conjecture}[section]

\theoremstyle{plain}
\newtheorem{question}{Question}[section]

\title{Punctured JSJ tori and tautological extensions of Azumaya algebras}
\author{Yi Wang}
\date{}

\begin{document}
	
\maketitle
	
\begin{abstract}
The $SL_2(\mathbb{C})$ character variety $X(M)$ has emerged as an important tool in studying the topology of hyperbolic 3-manifolds. In \cite{Chinburg2017AzumayaAA}, Chinburg-Reid-Stover constructed arithmetic invariants stemming from a canonical Azumaya algebra over the normalization of an irreducible component of $X(M)$ containing a lift of the holonomy representation of $M$. We provide an explicit topological criterion for extending the canonical Azumaya algebra over an ideal point, potentially leading to finer arithmetic invariants than those derived in \cite{Chinburg2017AzumayaAA}. This topological criterion involves Culler-Shalen theory \cite{Culler1983VarietiesOG} and, in some cases, JSJ decompositions of toroidal Dehn fillings of knot complements in the three-sphere. Inspired by the work of Paoluzzi-Porti \cite{Paoluzzi2010ConwaySA} and Tillmann \cite{Tillmann2003VarietiesAT}, \cite{Tillmann2005DegenerationsOI}, we provide examples of several cases where these refined invariants exist. Along the way, we show that certain families of Seifert surfaces in hyperbolic knot complements can be associated to ideal points of character varieties.
\end{abstract}

\section{Introduction}


The \emph{trace field} and \emph{quaternion algebra} associated to a hyperbolic knot $M$ are two important invariants (see \cite{Maclachlan2002TheAO}). Chinburg-Reid-Stover \cite{Chinburg2017AzumayaAA} studied the arithmetic of a \emph{canonical quaternion algebra} $A_{k(C)}$ that expands the quaternion algebra over entire components of the $SL_2(\mathbb{C})$ character variety $X(M)$. 

\medskip

An \emph{Azumaya algebra} over a scheme is a generalization of a central simple algebra over a field $k$. As discussed in Chapter 4 of \cite{Milne1980EtaleC}, Azumaya algebras are intimately connected with etale cohomology and the theory of \emph{Brauer groups} (i.e. the group of Azumaya algebras over a scheme with respect to a certain equivalence relation). Brauer groups arise in number theory and arithmetic geometry in many ways. They play a central role in local and global class field theory because of their connection to the second cohomology groups of the multiplicative groups of local and global fields \cite{Milne1980EtaleC}, \cite{SerreLocalFields}. The Brauer group of a smooth projective variety over a field is a birational invariant, and this can be used to show that various varieties are not rational \cite{ColliotThelene}. Artin conjectured that every proper scheme over the integers has a finite Brauer group (\cite{Milne1980EtaleC}, Question IV.2.19). This generalizes the conjectured finiteness of the so-called Tate-Shafarevitch group of the Jacobian of curves over number fields \cite{LeGroupDeBrauerIII}. Manin used the Brauer group to define the Brauer-Manin obstruction to the existence of rational points on smooth projective varieties over number fields \cite{Lang}. This obstruction has been used to prove a number of important results such as the vanishing of all Massey products of order greater than $2$ in the Galois cohomology of number fields \cite{HarpazWittenberg}. 

\medskip

In the realm of 3-manifold topology, the Chinburg-Reid-Stover invariant \cite{Chinburg2017AzumayaAA} can be viewed as either an equivalence class of Azumaya algebras or a 2-torsion Brauer group element of a projectivized normalized irreducible component $\tilde{C}$ of $X(M)$. This invariant is then shown to contain information about various aspects of hyperbolic 3-manifolds, such as ramification of quaternion algberas associated to hyperbolic Dehn fillings, orderability of $\pi_1(M)$, and $SU(2)$ representations. 

\medskip

A central question addressed in \cite{Chinburg2017AzumayaAA} was to find maximal open subsets $U$ of $\tilde{C}$ over which $A_{k(C)}$ can be extended to an Azumaya algebra $A_U$. The class $[A_U]$ in the Brauer group $Br(U)$ is determined by the class $[A_{k(C)}]$ in $Br(k(C))$, and that $[A_{k(C)}]$ determines and is determined by the isomorphism class of $A_{k(C)}$ as a quaternion algebra over $k(C)$. However, $[A_U]$ need not determine the isomorphism class of $A_U$. In this sense the work in \cite{Chinburg2017AzumayaAA} is not explicit, since it does not specify the isomorphism class of $A_U$. Rather, a purely algebraic argument involving tame symbols is used to show there is some $A_U$ defining a class $[A_U]$ that has image $[A_{k(C)}]$ in $Br(k(C))$.

\medskip

In order to refine the invariants constructed in \cite{Chinburg2017AzumayaAA}, we construct explicit extensions of $A_{k(C)}$ over ideal points that involve the topology of $M$. We define an extension of $A_{k(C)}$ over ideal points called \emph{tautological extension}, which in some cases comes directly from geometric decompositions of exceptional Dehn fillings of the knot complement. We prove the following theorem. See Section \ref{sec:background} for detailed definitions of the terms.

\begin{restatable}{thm}{main}\label{thm:main}
Let $M = S^3 \setminus K$ be a hyperbolic knot complement, and let $\{T_i\}_{i=1}^n \subset M$ be a system of disjoint non-parallel once-punctured JSJ tori with slope 0. Suppose the following:
\begin{itemize}
	\item All irreducible components of $X(M)$ containing irreducible characters are norm curves.
	\item No JSJ complementary region of $M(0)$ is Seifert fibered over the annulus.
	\item No JSJ torus in $M(0)$ bounds two hyperbolic components on both sides.  
	\end{itemize}
Then $\bigcup_{i=1}^nT_i$ is detected by an ideal point $x$ on a norm curve $\tilde{C} \subset \widetilde{X(M)}$, and $A_{k(C)}$ tautologically extends over $x$.
\end{restatable}

This theorem establishes a connection between the topology of some hyperbolic knot complements and the arithmetic of their character varieties while refining the invariant constructed in \cite{Chinburg2017AzumayaAA}. The underlying mechanism of the theorem is that hyperbolic holonomies of JSJ complementary regions of toroidal Dehn fillings of hyperbolic knots are reflected in ideal points of character varieties. This phenomenon has been exhibited in previous results of Tillmann \cite{Tillmann2005DegenerationsOI} and Paoluzzi-Porti \cite{Paoluzzi2010ConwaySA}. 

\medskip

We also demonstrate an infinite family of examples where Theorem \ref{thm:main} applies. Let $J(b_1, b_2)$ be the infinite family of two-bridge knots studied in \cite{Macasieb2009OnCV}.

\begin{corollary}\label{cor:twobridge}
Theorem \ref{thm:main} applies to the Seifert surfaces of the complements of $K = J(b_1, b_2)$ where:
\begin{enumerate}
	\item $|b_1|, |b_2| > 2$ are even.
	\item $|b_1| = 2$, $|b_2| > 2$ is even.
	\item $b_1$ is odd and $|b_2| = 2$. These are the twist knots.
\end{enumerate}
\end{corollary}

By the work in \cite{Patton1995IncompressiblePT}, this family of knots is exactly the genus one nonfibered hyperbolic two-bridge knots. 

\subsection{Outline of the paper}

Section \ref{sec:background} covers the necessary background to understand the methods and theorems in this paper. Section \ref{sec:lemmas} establishes various definitions and lemmas for the proof of Theorem \ref{thm:main}. In particular, we define the notions of tautological Azumaya extension and the limiting character, then prove numerous lemmas that are crucial to the proof of Theorem \ref{thm:main}. Section \ref{sec:proof} contains the proof of Theorem \ref{thm:main}. Section \ref{sec:conjectures} provides the proof of Corollary \ref{cor:twobridge} and poses further conjectures to generalize Theorem \ref{thm:main}.

\subsection{Acknowledgements}

The author would like to thank Ted Chinburg, Matthew Stover, and Dave Futer for helpful discussions. The author would also like to thank Ted Chinburg, Tam Cheetham-West, and Khanh Le for their comments on an earlier draft.

\section{Background}\label{sec:background}

\subsection{Character varieties and Culler-Shalen theory}

In this section, we will follow \cite{Culler1983VarietiesOG}.

\medskip

One method of studying hyperbolic 3-manifolds comes from \emph{character varieties} of $\pi_1(M)$, where $M$ is a hyperbolic 3-manifold. Let $\mathbb{H}^3$ be the hyperbolic upper half space, i.e. 
\begin{equation}
	\mathbb{H}^3 = \{(x, y, z) \in \mathbb{R}^3 \mid z > 0\}
\end{equation}
with metric
\begin{equation}
	ds^2 = \frac{dx^2+dy^2+dz^2}{z^2}
\end{equation}
Note that $M$ is homeomorphic to $\mathbb{H}^3/\Gamma$, where $\Gamma \cong \pi_1(M)$ is a discrete subgroup of the group $\text{Isom}^+(\mathbb{H}^3) \cong PSL_2(\mathbb{C}) = SL_2(\mathbb{C})/\{\pm I\}$. 

\begin{definition}
The discrete faithful embedding $\rho_0: \pi_1(M) \to PSL_2(\mathbb{C})$ is the \emph{holonomy representation} of $M$. 
\end{definition}

The space of traces of representations $\pi_1(M) \to PSL_2(\mathbb{C})$, denoted $Y(M)$, is an affine algebraic set.

\begin{definition}
Call $Y(M)$ the \emph{$PSL_2(\mathbb{C})$-character variety of $M$}.
\end{definition}  

Locally around $\rho_0$, this variety can be viewed as the space of \emph{deformations} of the holonomy representation. It is a consequence of Mostow's rigidity theorem that for a closed hyperbolic 3-manifold $M$, $Y(M)$ is a finite discrete set of points. Note that in general $Y(M)$ is not necessarily irreducible as an algebraic set. 

\begin{definition}
Let $Y_0(M)$ denote the irreducible component of the $PSL_2(\mathbb{C})$ character variety containing the holonomy representation, called the \emph{$PSL_2(\mathbb{C})$-canonical component}.
\end{definition} 

For a hyperbolic 3-manifold with boundary the union of tori, the complex dimension of $Y_0(M)$ is equal to the number of torus boundary components, by a theorem of Thurston.

\medskip

It is often convenient to work with $SL_2(\mathbb{C})$ as opposed to $PSL_2(\mathbb{C})$. One reason is that $SL_2(\mathbb{C})$ actually consists of matrices (as opposed to equivalence classes of matrices), which allows one to work with algebraic concepts such as traces and quaternion algebras, as we will later see. Thurston showed that the holonomy representation of any hyperbolic 3-manifold can be lifted to $SL_2(\mathbb{C})$. This argument was detailed in \cite{Culler1983VarietiesOG}. The space of traces of representations $\pi_1(M) \to SL_2(\mathbb{C})$ is parameterized by traces of combinations of generators of $\pi_1(M)$, and is also an algebraic set. 

\begin{definition}
Call this the \emph{$SL_2(\mathbb{C})$ character variety of $M$}, denoted $X(M)$.
\end{definition} 

Culler showed in \cite{Culler1986LiftingRT} that the entire $PSL_2(\mathbb{C})$ canonical component lifts to an irreducible component of the $SL_2(\mathbb{C})$ character variety, which is called the \emph{canonical component of $M$}, and will be denoted $X_0(M)$.

\medskip

It is well-known that the lift of the holonomy representation from $PSL_2(\mathbb{C})$ to $SL_2(\mathbb{C})$ is not unique; there are various interpretations of the number of such lifts. The one relevant to this paper is that given a representation $\rho: G \to PSL_2(\mathbb{C})$, the number of lifts into $SL_2(\mathbb{C})$ are in bijection with the first group cohomology $H^1(G; \mathbb{Z}/2\mathbb{Z})$, i.e. homomorphisms $\pi_1(M) \to \mathbb{Z}/2\mathbb{Z}$. These homomorphisms signify how one can flip signs of lifts to $SL_2(\mathbb{C})$ of $PSL_2(\mathbb{C})$ coming from $G$. The (co)-homological interpretation of these lifts will be pivotal to the discussion of $SL_2(\mathbb{C})$-compatibility in Section \ref{sec:lemmas}. 

\medskip

The work of many authors have established $Y(M), Y_0(M), X(M)$, and $X_0(M)$ as important tools in studying the topology of hyperbolic 3-manifolds. In particular, the work of Culler-Shalen (\cite{Culler1983VarietiesOG}, \cite{Culler1984BoundedSI}) connected $SL_2(\mathbb{C})$-character varieties with the study of \emph{essential surfaces}, i.e. embedded two-sided $\pi_1$-injective surfaces that are not boundary-parallel. We provide a brief summary here.

\medskip

Let $C$ be a complex 1-dimensional curve in $X(M)$, let $\overline{C}$ be the projective completion of $C$, and let $\tilde{C}$ be the normalization of $\overline{C}$. An \emph{ideal point} of $C$ is an element $x \in \tilde{C}$ that is birationally identified with a point in $\overline{C} \setminus C$; it can informally be described as a ``point at infinity". The work in \cite{Culler1983VarietiesOG} describes the following process to attribute an essential surface to any ideal point $x \in \widetilde{C}$.
\begin{itemize}
	\item Let $D$ be a curve lying in the projectivized $SL_2(\mathbb{C})$ representation variety that maps to the component $\overline{C}$ via the trace function. Then $P_C: \pi_1(M) \to SL_2(F)$, for $F = k(D)$, is the \emph{tautological representation} 
	\begin{equation}
		P_C(g) = \begin{pmatrix} f_{11}(g) & f_{12}(g) \\ f_{21}(g) & f_{22}(g)\end{pmatrix} \in SL_2(F)
	\end{equation}
	where $f_{ij}(g): D \to \mathbb{C}$ is a function that, evaluated at $\rho \in D$, outputs the $ij$th entry of $\rho(g)$, for $\rho \in D$. Note that $P_C$ is associated to the generic point of $D$.
	\item An ideal point $x$ gives rise to a discrete valuation $v_x$ on the function field of $C$ (defined by the polarity of the function at $x$), which extends to a valuation $v_{\tilde{x}}$ on $F$. Here $\tilde{x} \in \tilde{D}$ is an ideal point of $D$ lying over $x$ via the trace map.
	\item The valuation $v_{\tilde{x}}$ defines a tree called the \emph{Bass-Serre tree}. Roughly, the vertices of the tree are lattices inside $F^2$ which are integral with respect to $v_{\tilde{x}}$ up to scaling, and two vertices are connected by an edge if there exist lattice representatives $L$ and $L'$ such that $\pi L \subsetneq L' \subsetneq L$, where $\pi$ generates the valuation ring of $v_{\tilde{x}}$, denoted $\mathcal{O}_{v_{\tilde{x}}}$.
	\item The tautological representation $P_C$ gives rise to a group action of $\pi_1(M)$ on the Bass-Serre tree. This action satisfies the property that no edges are reversed and that stabilizers of vertices are conjugate to a subgroup of $SL_2(\mathcal{O}_{v_{\tilde{x}}})$.
	\item Consider an equivariant map from the universal cover of $M$ to the tree; an essential surface is obtained by taking the inverse image of the midpoints of edges of the tree and taking an appropriate surface in the isotopy class.
\end{itemize}

\begin{definition}
We call an essential surface arising from this construction a \emph{detected essential surface}.
\end{definition}

If the detected essential surface $S$ has boundary, $\partial S$ is a curve lying on the boundary torus which corresponds to an element of $\pi_1(\partial M) = \pi_1(T^2)$. 

\begin{definition}
Let $\partial S \in \pi_1(T^2)$ be the \emph{boundary slope} associated to $S$.
\end{definition}

It is a subtle problem to determine which essential surfaces and boundary slopes are detected by ideal points of character varieties. For instance, in \cite{Chesebro2005NotAB} it was shown that there exist boundary slopes in hyperbolic knot complements that are not detected by ideal points in the character variety. In \cite{Paoluzzi2010ConwaySA}, certain families of essential Conway spheres in hyperbolic knot complements were shown to always be detected by some ideal point in the character variety.

\subsection{Azumaya algebras}

We first define a special type of component of the character variety $X(M)$.

\begin{definition}\cite{Boyer2002OnTA}
A \emph{norm curve} is a component $C$ of the $SL_2(\mathbb{C})$ character variety such that given any nontrivial peripheral torus group element $g \in \pi_1(\partial M)$, $\text{tr}(g)$ is nonconstant on $C$.
\end{definition}

Norm curves were used by Boyer-Zhang in the proof of the finite filling conjecture \cite{Boyer2001APO}. In \cite{Chinburg2017AzumayaAA}, Chinburg-Reid-Stover defined arithmetic invariants of hyperbolic knots via the \emph{canonical Azumaya algebra} on $X_0(M)$. In fact, by Remark 1.3 in \cite{Chinburg2017AzumayaAA}, the invariants constructed in that paper exist over any norm curve. We will work with norm curves for the bulk of this paper, though the case of $X_0(M)$ (which is a norm curve by Proposition 2 in \cite{Culler1984BoundedSI}) is the most interesting.

\medskip

Let $C$ be any norm curve of a one-cusped hyperbolic 3-manifold $M$, $\overline{C}$ the projectivization. Let $F = k(D)$ be the function field of a curve $D$ lying over $\tilde{C}$. Recall the tautological representation $P_C: \pi_1(M) \to SL_2(F)$ defined in the previous section. Then
\begin{equation}
	A_{k(C)} = \left\{\sum_{i=1}^n\alpha_iP_C(\gamma_i) \mid \alpha_i \in k(C), \gamma_i \in \pi_1(M) \right\}
\end{equation}
is the \emph{canonical quaternion algebra} \cite{Chinburg2017AzumayaAA}. By the work of \cite{Maclachlan2002TheAO}, $A_{k(C)}$ is a four-dimensional central simple algebra over the function field $k(C)$, hence the name. Informally, an \emph{Azumaya algebra} is the generalization of the notion of a central simple algebra to rings, as opposed to fields (see \cite{Milne1980EtaleC} or Section 3 of \cite{Chinburg2017AzumayaAA} for more background).

\begin{definition}
Let $R$ be a ring with residue field $k$. An algebra $A$ over $R$ is an \emph{Azumaya algebra over $R$} if $A$ is free of finite rank $r \geq 1$ as an $R$-module, and $A \otimes k$ is a central simple algebra.
\end{definition}

One of the main problems in \cite{Chinburg2017AzumayaAA} is determining when $A_{k(C)}$ can be viewed as an Azumaya algebra over local rings of points on the canonical component. 

\begin{definition}
Let $x \in \widetilde{C}$ be a codimension one point. We say that $A_{k(C)}$ \emph{extends over $x$} if there exists an Azumaya algebra $A_x$ over $\mathcal{O}_x$ (the local ring at $x$) such that $A_{k(C)} \cong A_x \otimes_{\mathcal{O}_x} k(C)$. 
\end{definition}

On points corresponding to irreducible representations, say $\chi_\rho \in C$ with $\rho$ irreducible, the proof of Proposition 4.1 in \cite{Chinburg2017AzumayaAA} shows that $A_{k(C)}$ extends over $\chi_\rho$. The proof that $A_{k(C)}$ extends over ideal points of $C$ relies on tame symbol arguments, which produce equivalence classes of Azumaya algebras. This construction is entirely algebraic.

\medskip

One motivating question for the results in this paper is to find an explicit canonical extension of $A_{k(C)}$ over ideal points, as opposed to showing an equivalence class exists. In addition, we would like this canonical extension to be related to the geometry and topology of $M$. This would be a refinement of the invariant constructed in \cite{Chinburg2017AzumayaAA}, in the following sense. Let $X$ be a scheme. By the proof of Theorem 2.5 in \cite{Milne1980EtaleC}, the set of isomorphism classes of quaternion Azumaya algebras over $X$ is in bijection with $\check{H}^1(X_{et}, PSL_2)$. From the exact sequence of sheaves 
\begin{equation}
	1 \to \mu_2 \to SL_2 \to PSL_2 \to 1
\end{equation}
we get the cohomology exact sequence
\begin{equation}
	\check{H}^1_{et}(X, SL_2) \to \check{H}^1_{et}(X, PSL_2) \to \check{H}^2_{et}(X, \mu_2)
\end{equation}
Let $C$ be a norm curve in a character variety $X(M)$, and let $U \subset \tilde{C}$ be the Zariski-open subset consisting of points not associated to traces of reducible representations. Our goal is to produce an explicit cohomology element $\check{H}^1_{et}(U, PSL_2)$.

\medskip

The proof that $A_{k(C)}$ extends over points corresponding to irreducible representations, which is part of the proof of Proposition 4.1 in \cite{Chinburg2017AzumayaAA}, directly uses the image of the tautological representation to produce an extension of $A_{k(C)}$ over that point. We would like to use the same construction over ideal points. This leads to the following definition:

\begin{definition}
Let $x \subset C$ be a codimension one point. We say that $A_{k(C)}$ \emph{tautologically extends over $x$} if there exists $g, h \in \pi_1(M)$ such that the $\mathcal{O}_x$-span of $\{1, P_C(g), P_C(h), P_C(gh)\}$ is an Azumaya algebra over $\mathcal{O}_x$. 
\end{definition}

Tautological extension is an explicit and more refined extension of $A_{k(C)}$ over codimension-one points of the canonical component than the tame symbol arguments of \cite{Chinburg2017AzumayaAA}. The proof of Proposition 4.1 in \cite{Chinburg2017AzumayaAA} shows that for $\chi_\rho \in C$ the trace of an irreducible representation, $A_{k(C)}$ tautologically extends over $\chi_\rho$. The main question this paper addresses is the following:

\begin{question}\label{que:main1}
When does $A_{k(C)}$ tautologically extend over ideal points of the canonical component?
\end{question}

\subsection{Toroidal Dehn fillings and JSJ decompositions}

\begin{definition}
Given a 3-manifold $M$ with torus boundary $T^2$, the \emph{Dehn filling with slope $\beta \in \pi_1(\partial M)$} is the 3-manifold obtained by gluing a solid torus $S^1 \times D^2$ such that the gluing map sends $[\partial D^2] \in \pi_1(S^1 \times D^2)$ to $[\beta] \in \pi_1(\partial M)$. As an alternate notation, given a basis $\langle m, \ell \rangle$ of $\pi_1(T^2)$ and $p/q \in \mathbb{Q}$, the Dehn filling $M(p/q)$ is equal to $M(pm+q\ell)$. 
\end{definition}

\begin{definition}
Given a hyperbolic knot complement $N = M \setminus K$, an \emph{exceptional Dehn filling} $N(\beta)$ is a Dehn filling with slope $\beta$ such that $N(\beta)$ is non-hyperbolic.
\end{definition} 

It is a foundational theorem of Thurston \cite{Thurston1979TheGA} that any hyperbolic knot admits only finitely many exceptional Dehn fillings. Some exceptional Dehn fillings are \emph{toroidal}, i.e. contain an incompressible torus. For instance, \cite{Gabai} shows that 0-surgeries of genus one knots are toroidal, with the incompressible tori of $M(0)$ being capped-off genus one Seifert surfaces. The results in this paper deal largely with these toroidal Dehn fillings of hyperbolic knots, which can be decomposed in an understandable way. It turns out that tautological extension of Azumaya algebras over ideal points in character varieties is intimately connected with the geometry of these decomposed toroidal Dehn fillings, which we now discuss further.

\begin{definition}
A \emph{2-orbifold} is a quotient of $\mathbb{R}^2$ by a smooth properly discontinuous (but not necessarily free) group action $G$, where the action of $G$ has fixed points.
\end{definition}

Like 2-manifolds, 2-orbifolds have an Euler characteristic and obey a Gauss-Bonnet type result, and thus can be divided into spherical, Euclidean, and hyperbolic geometries. See \cite{Thurston1979TheGA} for more details.

\begin{definition}
A manifold $M$ is \emph{Seifert-fibered} over a 2-orbifold $S$ if there exists a fiber bundle $S^1 \hookrightarrow M \to S$. The generator of $\pi_1(S^1)$ in $\pi_1(M)$ is called the \emph{regular fiber}.
\end{definition}

See \cite{Jaco1980LecturesOT} for more on Seifert-fibered manifolds. 

\begin{definition}
A \emph{JSJ decomposition} is a minimal collection of disjointly embedded incompressible tori $\{T_i\} \subset M$ such that each connected component $M_i \subset \bigcup_{i=1}^nT_i$ is atoroidal (i.e does not contain an essential torus) and either hyperbolic or Seifert-fibered.
\end{definition}

In \cite{Jaco1979SeifertFS}, Jaco-Shalen-Johannson showed that every closed 3-manifold admits a JSJ decomposition. 

\medskip

By Lemma 3 in \cite{Hatcher1985IncompressibleSI}, we have the following fact: given a knot complement $N = M - K$ and a Dehn filling $N(\beta)$, any incompressible surface in $N(\beta)$ must come from an incompressible surface in $N$. This means that a JSJ decomposition in a toroidal Dehn filling $N(\beta)$ of a hyperbolic knot complement must come from a disjoint union of incompressible $n$-punctured tori in $N$ (the tori must be punctured, or else $M - K$ would not be a hyperbolic knot complement). Note that the tori can have different numbers of punctures. 

\begin{definition}
We will refer to such a system of tori as a system of \emph{punctured JSJ tori}.
\end{definition} 

Once-punctured JSJ tori, i.e. Seifert surfaces of genus one, have been studied in \cite{ValdezSnchez2019SeifertSF} and \cite{Tsutsumi2003UniversalBF}, mainly in the context of bounding the number of JSJ tori in toroidal Dehn fillings. In addition, twice-punctured tori were studied in \cite{Gordon2000DehnSO}, \cite{Patton1995IncompressiblePT}. We will study incompressible punctured JSJ tori, particularly genus one Seifert surfaces, in the context of Culler-Shalen theory and Azumaya algebras. Upcoming work \cite{newpaper} will address the case of twice-punctured tori in certain knot complements. 

\subsection{Outline of the proof}

Recall the main theorem.

\main*

This is analogous to the main theorem of \cite{Paoluzzi2010ConwaySA}, in which Bonahon-Siebenmann systems of Conway spheres in hyperbolic knot complements are shown to be detected by ideal points of character varieties. The proof of the theorem combines techniques of Paoluzzi-Porti from \cite{Paoluzzi2010ConwaySA} and Tillmann from \cite{Tillmann2003VarietiesAT}. In analogy with Lemma 10 of \cite{Paoluzzi2010ConwaySA}, the proof also relies on \emph{$SL_2(\mathbb{C})$-compatibility} of JSJ decompositions, which means that holonomy representations of JSJ complementary components can be lifted to $SL_2(\mathbb{C})$ in such a way that the traces match on the torus boundary components. This property is proven in our specific case in Section \ref{sec:lemmas}; the proof relies on a deformation theory argument from \cite{MenalFerrer2010TwistedCF} and an analysis of the $\mathbb{Z}/2$-homology structure of the JSJ complementary regions. 

\medskip

By the way tautological extension is defined, it is directly related to the limit of the trace of tautological representation at the ideal point, which Tillmann defines as the \emph{limiting character}. One core idea of this paper is that the limiting character often reflects the geometry of complementary regions of capped-off surfaces in Dehn fillings, which manifests in the results of \cite{Paoluzzi2010ConwaySA} and Theorem \ref{thm:main}. In particular, in the situation of Theorem \ref{thm:main}, the fact that $A_{k(C)}$ tautologically extends over an ideal point can be attributed to the limiting character at that ideal point being equal to the trace of a holonomy representation of a JSJ complementary region of $M(0)$.

\medskip

The sketch of the proof is as follows. 
\begin{enumerate}
	\item Notice that the JSJ complementary regions of $M(0)$ are arranged in a circle, and each region thus has two boundary tori. (Lemma \ref{lma:verticesedges})
	\item Do a complex dimension count to show that each JSJ complementary region of $M(0)$ is either hyperbolic or Seifert-fibered over an annulus with a cone point of order $p \geq 2$, also called a \emph{cable space}. (Corollary \ref{cor:cablespaces})
	\item Show that there exist lifts of the holonomy representations of each JSJ complementary component so that the signs of the traces match on the boundary tori. (Theorem \ref{thm:compatible})
	\item Use the properties of JSJ decompositions (Theorem \ref{thm:JSJ}) and the gluing data to show that the holonomy representations restricted to the boundary tori have the same trace, but are nonconjugate. This is only possible because these representations are reducible on the boundary tori.
	\item Argue as in \cite{Paoluzzi2010ConwaySA} to show that lifts of the holonomy representations to $SL_2(\mathbb{C})$ derived in the third step must correspond to ideal points in the character variety. Intuitively, the argument uses the fact that on the boundary tori of the JSJ components, the holonomy characters are equal but the holonomy representations are nonconjugate. Since the traces can be glued, it corresponds to the limit of a sequence of points on the character variety, but since the traces are nonconjugate, this limit cannot correspond to the trace of a representation of $\pi_1(M)$, and hence must correspond to an ideal point.
	\item Use a result of Tillmann (Lemma \ref{lma:tillmann}) to conclude that the surface is detected by an ideal point on a character variety component containing irreducible characters, hence a norm curve.
	\item Since the holonomy representations on the complementary components are irreducible, conclude that $A_{k(C)}$ extends over this ideal point. (Proposition \ref{prop:limitingrep})
\end{enumerate}

\section{Prerequisite lemmas}\label{sec:lemmas}


\subsection{Tautological extension}

Recall the definition of tautological Azumaya extension:

\begin{definition}
Let $C \subset X(M)$ be a norm curve, and let $x \subset \tilde{C}$ be a codimension one point. We say that $A_{k(C)}$ \emph{tautologically extends over $x$} if there exists $g, h \in \pi_1(M)$ such that the $\mathcal{O}_x$-span of $\{1, P_C(g), P_C(h), P_C(gh)\}$ is an Azumaya algebra over $\mathcal{O}_x$. 
\end{definition}

The proof of Proposition 4.1 in \cite{Chinburg2017AzumayaAA} shows that $A_{k(C)}$ tautologically extends over any $\chi = \text{tr}(\rho) \in C$ such that $\rho$ is irreducible. If $\rho$ is reducible, we have the following.

\begin{prop}
Let $x_w \in C$ be a point corresponding to a reducible representation $\rho_w$. Then $A_{k(C)}$ does not tautologically extend over $x_w$. 
\end{prop}

\begin{proof}
Suppose $A_{k(C)}$ has a tautological extension over $x_w$. Let $\mathcal{O}_{x_w}$ be the valuation ring associated to $x_w$, with residue field $k(x_w)$, and let $A_{x_w}$ be the $\mathcal{O}_{x_w}$-span of the basis $\{1, P_C(g), P_C(h), P_C(gh)\}$ for some $g, h$ which is an Azumaya algebra. Recall that
\begin{equation}
	A_{k(C)} = \left\{\sum_{i=1}^n\alpha_iP_C(\gamma_i) \mid \alpha_i \in k(C), \gamma_i \in \pi_1(M)\right\}
\end{equation}
Then $A_{x_w} \otimes k(x_w)$ must be the $k(x_w)$-span of the set $\{1, \rho_w(g), \rho_w(h), \rho_w(gh)\}$. Since $\rho_w$ is reducible, by Lemma 1.2.4 in \cite{Maclachlan2002TheAO}, these vectors are dependent, and hence cannot generate a 4-dimensional algebra over $k_w$. Thus this span is not a central simple algebra over $k_w$. This leads to a contradiction, and thus $A_{k(C)}$ cannot tautologically extend over $x_w$. 
\end{proof}

In light of this, the ideal points $x \in \widetilde{C}$ are the only points where tautological extension of $A_{k(C)}$ is to be determined. 

\subsection{The limiting character}

Let $D$ be a curve in the projectivized representation variety lying over $\overline{C}$. Recall the tautological representation $P_C: \pi_1(M) \to SL_2(F)$, with $F = k(D)$. Now suppose $S$ is an essential surface detected by an ideal point $x \in \widetilde{C}$, with associated valuation ring $\mathcal{O}_x$. Let $\tilde{x}$ be an ideal point in $D$ lying over $x$. Note that since $x$ is an ideal point, there will exist some $h \in \pi_1(M)$ such that $P_c(h)$ cannot be conjugated into $SL_2(\mathcal{O}_{\tilde{x}})$, i.e. some trace functions will have poles at the ideal point. Let $M_1, \dots, M_n$ be the components of $\overline{M \setminus S}$, i.e. the complementary regions. Both $\pi_1(S)$ and $\pi_1(M_i)$ can be viewed as subgroups of $\pi_1(M)$, since $S$ is essential. From Bass-Serre theory (as in \cite{Culler1983VarietiesOG}), for any $g \in \pi_1(M_i) \subset \pi_1(M)$, $P_c(g)$ is $GL_2(F)$-conjugate to an element of $SL_2(\mathcal{O}_{\tilde{x}})$, where $\mathcal{O}_{\tilde{x}} \subset  k(D)$ is the valuation ring associated to the ideal point $\tilde{x} \in D$.

\begin{definition}
The \emph{limiting character on $M_i$}, denoted $\chi_{\infty, i} = \text{tr}(\rho_{\infty, i})$, is the character of a representation
\begin{equation}
	\rho_{\infty, i}: \pi_1(M_i) \to SL_2(\mathbb{C})
\end{equation}
obtained as follows. Given $g \in \pi_1(M_i)$, take the matrix $P_c(g) \in SL_2(F)$, conjugate into $SL_2(\mathcal{O}_{\tilde{x}})$, and compose with the quotient map $\mathcal{O}_{\tilde{x}} \to k_v = \mathcal{O}_{\tilde{x}}/\mathfrak{m}_{\mathcal{O}_{\tilde{x}}}$ to obtain an element of $SL_2(k_v) \subset SL_2(\mathbb{C})$. Here $\mathfrak{m}_{\mathcal{O}_{\tilde{x}}}$ is the maximal ideal of the valuation ring. Since conjugation preserves trace, the trace of this element is well-defined, and we define $\chi_{\infty, i}(g)$ to be the trace of this matrix.
\end{definition}

\begin{remark}
In order for this definition to go through, we need that the entire subgroup $P_c(\pi_1(M_i))$ is $GL_2(F)$-conjugate to a subgroup of $SL_2(\mathcal{O}_{\tilde{x}})$. Since $\pi_1(M_i)$ is contained in a vertex stabilizer of the Bass-Serre tree, this follows from Theorem 2.1.2 in \cite{Culler1983VarietiesOG}.
\end{remark}

\begin{remark}
Let $\{\chi_j\} \subset C$ be a sequence of characters that approach $x$. One can see that
\begin{equation}
	\chi_{\infty, i} = \lim_{\chi_j \rightarrow x}\chi_j|_{\pi_1(M_i)}
\end{equation}
In other words, the limiting character is exactly what it sounds like.
\end{remark}

The limiting character takes on finite values on each connected component of $S$, since $\pi_1(S_i) \subset \pi_1(M_j)$ for each connected component $S_i$, and for some $M_j$. By an argument in Section 3.8 of the third chapter of \cite{Daverman2002HandbookOG}, we have the following:

\begin{prop}[\cite{Daverman2002HandbookOG}]
The limiting character is equal to the trace of a reducible representation on a detected essential surface at the ideal point.
\end{prop}

We have already seen that tautological extensions cannot come from reducible representations associated to any point, and so the basis for a tautological extension of $A_{k(C)}$ over an ideal point cannot come from the fundamental group of a detected essential surface. Instead, we will look toward its complement.

\begin{prop}\label{prop:limitingrep}
Let $S$ be an essential surface detected by the ideal point $x \in \widetilde{C}$. Then if there exists a component $M' \subset \overline{M \setminus S}$ such that the limiting character $\chi_{\infty}: \pi_1(M') \to \mathbb{C}$ is the trace of an irreducible representation, then $A_{k(C)}$ tautologically extends over $x$.
\end{prop}

\begin{proof}
Suppose that $\chi_\infty$ is the trace of an irreducible representation $\rho_\infty: \pi_1(M') \to SL_2(\mathbb{C})$. Let $k_\infty$ be the field generated by traces of elements in $\pi_1(M')$ under $\rho_\infty$. By the proof of Lemma 2.5 in \cite{Chinburg2017AzumayaAA}, $k_\infty$ is also equal to the residue field of $\widetilde{C}$ at $x$. Let 
\begin{equation}
	A_{\rho_\infty} = \left\{\sum_{i=1}^n\alpha_i\rho_\infty(\gamma_i) \mid \alpha_i \in k_\infty, \gamma_i \in \pi_1(M')\right\}
\end{equation}
Note that $A_{\rho_\infty}$ is a quaternion algebra over $k_\infty$ because $\rho_\infty$ is irreducible. Since $\rho_\infty$ is irreducible, it also follows that $\pi_1(M')$ is nonabelian. Choose noncommuting $g, h \in \pi_1(M')$. By Section 3.6 in \cite{Maclachlan2002TheAO}, $\{1, \rho_\infty(g), \rho_\infty(h), \rho_\infty(gh)\}$ is a $k_\infty$-basis for $A_{\rho_\infty}$. Then $\{1, P_C(g), P_C(h), P_C(gh)\}$ forms a basis for $A_{k(C)}$. Let $A_x$ be the $\mathcal{O}_x$-span of this basis. Then the reduction of $A_x$ modulo the maximal ideal $\mathfrak{m}_{\mathcal{O}_x}$ is equal to $A_{\rho_\infty}$, which is a quaternion algebra over $k_\infty$. Thus, $A_x \otimes k_\infty = A_x \otimes k(x) = A_{\rho_\infty}$ is a central simple algebra. It is immediate from our constructions that $A_x \otimes k(C) \cong A_{k(C)}$, so by Theorem 3.1.6 in \cite{Chinburg2017AzumayaAA}, $A_{k(C)}$ tautologically extends over $x$. 
\end{proof}

By the above proposition, the main question of this paper can be interpreted as:

\begin{question}\label{que:main2}
When is the limiting character $\chi_\infty$ irreducible at an ideal point of a norm curve?
\end{question}

\subsection{Cable spaces}

Theorem \ref{thm:main} is about JSJ decompositions of toroidal surgeries. Here we set up a discussion of $SL_2(\mathbb{C})$-compatiblity of the JSJ decompositions of these toridal Dehn fillings, which will be needed for the proof of Theorem \ref{thm:main}. 

\begin{definition}
Let $A^2(q)$ be a 2-orbifold with base space an annulus and a cone point of order $q$ for $q \geq 2$. A \emph{cable space} is a 3-manifold $N$ Seifert-fibered over $A^2(q)$. 
\end{definition}

Cable spaces will be the most common JSJ-components of toroidal Dehn fillings of knot complements with essential once-punctured tori. In order to prove Theorem \ref{thm:main}, we must analyze the character varieties of these spaces.

\begin{lemma}
Let $N$ be a cable space. Then $X(N)$ is two-dimensional.
\end{lemma}

\begin{proof}
By Lemma 3.1 in \cite{Kitano1994ReidemeisterTO}, every irreducible representation of $\pi_1(N)$ into $SL_2(\mathbb{C})$ sends the regular fiber to $\pm I$. It follows that the dimension of $X(N)$ is equal to the dimension of $X(A^2(q))$. We have the fundamental group presentation
\begin{equation}
	\pi_1(A^2(q)) = \langle a, b \mid (ab)^q = 1 \rangle
\end{equation}
Let $F_2 = \langle a, b \rangle$ be the free group on two generators. Knowing that $X(F_2) \cong \mathbb{C}^3$, it follows that $X(A^2(q)) = \{(x, y, z) \in \mathbb{C}^3 \mid z = \zeta_q + \zeta_q^{-1}\} \cong \mathbb{C}^2$, where $\zeta_q$ is any $q$th root of unity. This is two-dimensional, so we are done.
\end{proof}

\begin{lemma}\label{lma:twodim}
The only hyperbolic 2-orbifolds with two geodesic boundary components whose character variety has dimension two are $A^2(q)$, where $A^2$ is the annulus and $q \geq 2$ is the order of the cone point. 
\end{lemma}

\begin{proof}
First, suppose the base space of a Seifert-fibered space is orientable, i.e. a twice-punctured orientable surface of genus $g$. Suppose we are in the situation of an orbifold whose base space is a twice-punctured surface $S_g$ with $c$ cone points. Call the orbifold $S_g'$. Let $S_{g,n}$ be the $n$-punctured surface; we have $\dim(X(S_g')) \geq \dim(S_{g,c+2}) - c$, since there are $c$ equations defining the relations induced by the cone points. Thus, it suffices to show that for all $S_{g,n}$ with $n \geq 2$ and not a twice or thrice-punctured sphere, $\dim(X(S_{g,n})) > n$. Note that $S_{g,n}$ is homotopy equivalent to a wedge of $n + 2g - 1$ circles, and so $\pi_1(S_{g,n}) \cong F_{n+2g-1}$, i.e. the free group on $n+2g-1$ generators. By \cite{Ashley2017Rank1C}, $\dim(X(F_{n+2g-1})) = 3n+6g-6$, and if $n \geq 2$ and $g > 0$ this is greater than $n$. If $g = 0$, $n > 3$, and so $3n+6g-6>n$ still. So we are done if the orbifold is orientable. 
	
\medskip
	
If the orbifold is non-orientable, the base orbifold is a twice-punctured connected sum of $g$ copies of $\mathbb{R}P^2$ with some cone points; let the base surface with $n$ punctures be denoted $S_{g,n}$ as above. By the same logic as above, it suffices to show that $\dim(X(S_{g,n})) > n$. Indeed, since $S_{g,n}$ is homotopy equivalent to a wedge of $n + g - 1$ circles, we have that $\pi_1(S_{g,n}) = F_{n+g-1}$, and from \cite{Ashley2017Rank1C} we have $\dim(X(F_{n+g-1})) = 3n+3g-6 > n$ if $n \geq 2$ and $g > 0$. So we are done in both the orientable and non-orientable cases now. 
\end{proof}

Combining with Lemma 3.1 in \cite{Kitano1994ReidemeisterTO} gives us the following:

\begin{corollary}\label{cor:cablespaces}
If an orientable 3-manifold with two torus boundary components is not Seifert-fibered over the annulus and has character variety of dimension two, it is either hyperbolic or a cable space.
\end{corollary}

\subsection{$SL_2(\mathbb{C})$-compatibility}

To prove the main theorem, we must establish a property for JSJ decompositions, defined as follows:

\begin{definition}
Suppose we are given a 3-manifold $N$ with JSJ composition given by tori $T_1, \dots, T_n$, and connected components $N_1, \dots, N_r$. We say that $\{T_i\}$ is a \emph{$SL_2(\mathbb{C})$-compatible JSJ decomposition} if there exist representations $\rho_i: \pi_1(N_i) \to SL_2\mathbb{C}$ such that the following hold.
\begin{enumerate}
	\item For any $g \in \pi_1(\partial N_j)$, $\text{tr}(\rho_j(g)) = \pm 2$. 
	\item Let $T_i', T_j'$ be torus boundary components of $N_i$ and $N_j$. If $\varphi: T_i' \to T_j'$ is a gluing homeomorphism, then for all $g \in \pi_1(\partial T_i')$, $\text{tr}(\rho_i(g)) = \text{tr}(\rho_j(\varphi(g)))$.
\end{enumerate}
\end{definition}

Informally, $SL_2(\mathbb{C})$-compatibility means that traces can be ``glued" together along their torus boundary components. 

\medskip

In the situation we deal with in the the main theorem, it will be crucial to establish $SL_2(\mathbb{C})$-compatibility of JSJ decompositions of toroidal Dehn fillings. We will see that the JSJ decompositions we deal with have components that are all either cable spaces or hyperbolic. Then the representations $\rho_i$ that realize the $SL_2(\mathbb{C})$-compatibility of these JSJ decompositions will be lifts of holonomy representations of either the underlying hyperbolic 2-orbifold or the hyperbolic structure on $N_i$. The traces of these irreducible holonomy representations will then be the limiting character of a detected surface, establishing tautological extension of $A_{k(C)}$ by Proposition \ref{prop:limitingrep}. 

\medskip

For hyperbolic components of JSJ decompositions, we have the following lemma from \cite{MenalFerrer2010TwistedCF}:

\begin{lemma}\label{lma:hyperbolicgoodgluing}\cite{MenalFerrer2010TwistedCF}
Given a hyperbolic 3-manifold $M$ with torus cusps $T_1, \dots, T_m$ and holonomy representation $\rho: \pi_1(M) \to PSL_2(\mathbb{C})$, and any $(\gamma_1, \dots, \gamma_m) \in \pi_1(T_1) \times \dots \times \pi_1(T_m)$ where the $\gamma_i$ are simple and nontrivial, there exits a lift of the holonomy representation
\begin{equation}
	\hat{\rho}: \pi_1(M) \to SL_2(\mathbb{C})
\end{equation}
such that 
\begin{equation}
	\text{tr}(\hat{\rho}(\gamma_i)) = -2 
\end{equation}
for $i = 1, \dots, m$. 
\end{lemma}

\begin{remark}\label{rmk:homology}
It is known that lifts of $PSL_2(\mathbb{C})$ representations are in bijection with $H^1(\pi_1(M); \mathbb{Z}/2\mathbb{Z})$, i.e. homomorphisms $\pi_1(M) \to \mathbb{Z}/2\mathbb{Z}$. It is exactly those elements that map to the identity under all homomorphisms whose traces are independent of the lift. (These elements turn out to be squares modulo the commutator subgroup of $\pi_1(M)$.) By the universal coefficient theorem this is isomorphic to $H_1(M; \mathbb{Z}/2\mathbb{Z})$, and so lifts of the holonomy representation into $PSL_2(\mathbb{C})$ to $SL_2(\mathbb{C})$ are in bijection with the singular $\mathbb{Z}/2\mathbb{Z}$ homology. Thus, Lemma \ref{lma:hyperbolicgoodgluing} implies that any $\mathbb{Z}/2\mathbb{Z}$-homologically trivial group element of $\pi_1(M)$ must lift to trace -2 in the holonomy representation of a hyperbolic 3-manifold. In addition the traces of any elements of $\pi_1(M)$ that are $\mathbb{Z}/2\mathbb{Z}$-homologous are the same for any lift of the holonomy representation.
\end{remark}

We now show that cable spaces have a similar property. We will make liberal use of the following theorem, which is colloquially known as the ``half lives, half dies" theorem. We will use the $\mathbb{Z}/2\mathbb{Z}$ coefficients version; this is Lemma 5.3 in  \cite{Schleimer2018IntroductionTT}.

\begin{theorem}[\cite{Schleimer2018IntroductionTT}, ``half lives, half dies"]
Suppose $M$ is a compact oriented connected 3-manifold. Let $i: \partial M \to M$ be the inclusion, inducing $i_*: H_1(\partial M; \mathbb{Z}/2\mathbb{Z}) \to H_1(M; \mathbb{Z}/2\mathbb{Z})$. Then
\begin{equation}
	\frac{1}{2}\dim(H_1(\partial M)) = \dim(\ker(i_*)) = \dim(\text{im}(i_*))
\end{equation}
\end{theorem}

\begin{lemma}\label{lma:cablespacegoodgluing}
Suppose $M$ is a cable space fibered over $A^2(q)$ with torus boundary components $T_1, T_2$. Let $\rho: \pi_1(M) \to PSL_2(\mathbb{C})$ be the holonomy representation of $A^2(q)$. Given and $(\gamma_1, \gamma_2) \in \pi_1(T_1) \times \dots \times \pi_1(T_m)$ with $\gamma_i$ simple and nontrivial, there exists a list of the holonomy representation 
\begin{equation}
	\widehat{\rho}: \pi_1(M) \to SL_2(\mathbb{C})
\end{equation}
such that 
\begin{equation}
	\text{tr}(\widehat{\rho}(\gamma_i)) = -2
\end{equation}
for $i = 1, 2$.
\end{lemma}

\begin{proof}
The holonomy representation of the thrice-punctured sphere, whose fundamental group is given by $\langle a, b, c \mid c = ab \rangle$, is 
\begin{equation}
	\rho(a) = \left[\pm\begin{pmatrix}1&2\\0&1\end{pmatrix}\right] \ \ \ \ \ \rho(b) = \left[\pm\begin{pmatrix}1&0\\-2&1\end{pmatrix}\right] \ \ \ \ \ \rho(c) = \left[\pm\begin{pmatrix}-3&2\\-2&1\end{pmatrix}\right]
\end{equation}
In particular, for any lift $\widehat{\rho}$ of this representation to $SL_2(\mathbb{C})$, if the signs of $\widehat{\rho}(a)$ and $\widehat{\rho}(b)$ are the same, then the sign of $\widehat{\rho}(c)$ is negative, and if the traces of $a$ and $b$ have different sign, then $c$ has positive sign. The holonomy representation of $A^2(q)$ is given by 
\begin{equation}
	\rho(a) = \left[\pm\begin{pmatrix}1&2\\0&1\end{pmatrix}\right] \ \ \ \ \ \rho(b) = \left[\pm\begin{pmatrix}1&0\\x_q&1\end{pmatrix}\right] \ \ \ \ \ \rho(c) = \left[\pm\begin{pmatrix}1 + 2x_q&2\\x_q&1\end{pmatrix}\right]
\end{equation}
where $\xi_{2q}$ is a primitive $2q$th root of unity and $x_q$ is such that
\begin{equation}
	2 + 2x_q = -(\xi_{2q} + \xi_{2q}^{-1})
\end{equation}
Since the holonomy of $A^2(q)$ is a continuous deformation of the holonomy of the thrice-punctured sphere, it follows that the signs of the traces of a lift of $a$ and $b$ are equal if and only if $c$ lifts to a negative trace, as is true for the thrice-punctured sphere. 
	
\medskip
	
We divide into three cases. Let $\pi_1(A^2(q)) = \langle a, b, c \mid c = ab, c^q = 1 \rangle$. The strategy will be to determine the kernel of $H_1(\partial M; \mathbb{Z}/2\mathbb{Z}) \to H_1(M; \mathbb{Z}/2\mathbb{Z})$, and use the homological interpretation of the kernel in Remark \ref{rmk:homology} to reach our conclusion.
	
\medskip
	
\underline{Case 1: $q$ is even.} In this case, we have $c^qh^p = 1$ for some $p$ which must be odd. The holonomy representation $\rho$ satisfies
\begin{equation}
	\rho(c) = \left[\pm\begin{pmatrix}\xi_{2q}&0\\0&\xi_{2q}^{-1}\end{pmatrix}\right]
\end{equation}
where $\xi_{2q}$ denotes a primitive $2q$th root of unity. In this case, since $q$ is even, given any lift $\widehat{\rho}: \pi_1(M) \to SL_2(\mathbb{C})$, $\widehat{\rho}(c^q) = -I$. By Lemma 3.1 in \cite{Kitano1994ReidemeisterTO}, we know that $\rho(h) = [\pm I]$, and the relation $c^qh^p = 1$ implies that for all lifts $\widehat{\rho}$, $\widehat{\rho}(h) = -I$ as well. Let $h_j \in H_1(T_j; \mathbb{Z}/2\mathbb{Z})$ be the singular homology classes of $h$ in the boundary tori of $M$. Since the sign of $\text{tr}(\widehat{\rho}(h))$ is fixed for all lifts $\widehat{\rho}$, both $h_j$ lie in the kernels of $H_1(T_j; \mathbb{Z}/2\mathbb{Z}) \to H_1(M; \mathbb{Z}/2\mathbb{Z})$. By the ``half lives, half dies" theorem, $h_1, h_2$ generate the kernel of $H_1(\partial M; \mathbb{Z}/2\mathbb{Z}) \to H_1(M; \mathbb{Z}/2\mathbb{Z})$. In particular, the signs of the traces of pairs of other elements in $\pi_1(T_j)$ can be changed independently, which proves the lemma in this case.
	
\medskip
	
\underline{Case 2: $q$ is odd, $c^qh^p = 1$ with $p$ even.} In this case, the sign of the trace of $c$ must be negative for all lifts $\widehat{\rho}$, i.e.
\begin{equation}
	\widehat{\rho}(c) = \begin{pmatrix}-\xi^{2q} & 0 \\ 0 & -\xi_{2q}^{-1}\end{pmatrix}
\end{equation}
This means that the sign of the traces of $a$ and $b$ are equal for all lifts. In addition, the traces of $h$ as an element of $\pi_1(T_1)$ and $h$ as an element of $\pi_1(T_2)$ are also equal. View $H_1(\partial M; \mathbb{Z}/2\mathbb{Z}) \cong H_1(T_1; \mathbb{Z}/2\mathbb{Z}) \oplus H_1(T_2; \mathbb{Z}/2\mathbb{Z})$. By the ``half lives, half dies" theorem, the kernel of $H_1(\partial M; \mathbb{Z}/2\mathbb{Z}) \to H_1(M; \mathbb{Z}/2\mathbb{Z})$ is generated by $[a] \oplus [b]$ and $[h] \oplus [h]$. In particular, there exist lifts of $\rho$ such that any two of $[a], [h], [ah]$ have trace $-2$, meaning that the same respective two of $[b], [h], [bh]$ also have trace $-2$, proving the lemma in this case.
	
\medskip
	
\underline{Case 3: $q$ is odd, $c^qh^p = 1$ with $p$ odd.} In this case, the signs of the traces of $c$ and $h$ are opposite for all lifts $\widehat{\rho}$. In particular, for all lifts, the trace of $ah$ and the trace of $b$ have the same sign, for if $h$ has positive sign, then $c$ has negative sign, and hence the traces of $a$ and $b$ have the same sign, and if $h$ has negative sign, then $c$ has positive sign, and the traces of $a$ and $b$ have opposite sign. By the ``half lives, half dies" theorem, the kernel of $H_1(\partial M; \mathbb{Z}/2\mathbb{Z}) \to H_1(M; \mathbb{Z}/2\mathbb{Z})$ is generated by $[ah] \oplus [b]$ and $[h] \oplus [h]$. In particular, there exist lifts of $\rho$ such that any two of $[a], [h], [ah]$ have trace $-2$, meaning that the same  respective two of $[bh], [h], [b]$ also have trace $-2$, proving the lemma in this case. 
\end{proof}

\begin{remark}
Lemmas \ref{lma:hyperbolicgoodgluing} and \ref{lma:cablespacegoodgluing} show that for any $M$ hyperbolic or a cable space, any $\gamma \in \pi_1(\partial M)$ whose trace is independent of the chosen lift $\hat{\rho}$ has trace -2. Otherwise, one may change the sign of the trace of $\gamma$ independent of the signs of the traces of the other $\mathbb{Z}/2\mathbb{Z}$ homology classes.
\end{remark}

\subsection{Results on $SL_2(\mathbb{C})$-compatibility}

\begin{theorem}\label{thm:compatible}
Let $M$ be an irreducible closed 3-manifold with a JSJ decomposition with JSJ tori $\{T_i\}_{i=1}^m$, gluing homeomorphisms $\{\varphi_i\}_{i=1}^m$, and complementary components $\{M_i\}_{i=1}^n$ such that:
\begin{enumerate}
	\item Each $M_i$ is either hyperbolic or a cable space.
	\item Each $M_i$ has exactly two torus boundary components $T_{i_1}, T_{i_2}$, and the JSJ complementary components are arranged in a circle.
\end{enumerate} 
Then the JSJ decomposition is $SL_2(\mathbb{C})$-compatible.
\end{theorem}

\begin{proof}
Take any $M_i$, and let $\rho_i$ be either the holonomy representation or the composition of the Seifert fibration map with the holonomy of $A^2(q)$. Let $i_*: H_1(\partial M_i; \mathbb{Z}/2\mathbb{Z}) \cong \mathbb{Z}^4 \to H_1(M_i; \mathbb{Z}/2\mathbb{Z})$ be the induced map of the boundary inclusion. By the ``half lives, half dies" theorem, the kernel is 2-dimensional, i.e.
\begin{equation}
	\ker(i_*) = \langle \gamma_1, \gamma_2 \rangle \cong (\mathbb{Z}/2\mathbb{Z})^2 \subset H_1(\partial M_i; \mathbb{Z}/2\mathbb{Z})
\end{equation}
We now analyze the possibilities for $\gamma_1, \gamma_2$. First, note that there are two possibilities for $\gamma_k$:
\begin{enumerate}
	\item $\gamma_k \in H_1(T_{i_j}; \mathbb{Z}/2\mathbb{Z})$ for some $j \in \{1, 2\}$. We say $\gamma_k$ is \emph{contained in $T_{i_j}$} in this case. By Lemmas \ref{lma:hyperbolicgoodgluing} and \ref{lma:cablespacegoodgluing}, the trace of any lift of $\rho_i$ evaluated on a representative of $\gamma_k$ must be -2.
	\item $\langle \gamma_k \rangle \cap H_1(T_{i_j}; \mathbb{Z}/2\mathbb{Z}) \neq 0$ for $j = 1, 2$. We say $\gamma_k$ is \emph{spanning} in this case. The trace of any lift of $\rho_i$ evaluated on representatives of $\langle \gamma_k \rangle \cap H_1(T_{i_1}; \mathbb{Z}/2\mathbb{Z})$ and $\langle \gamma_k \rangle \cap H_1(T_{i_2}; \mathbb{Z}/2\mathbb{Z})$ are equal, since they are $\mathbb{Z}/2\mathbb{Z}$-homologous. 
\end{enumerate}
Here are the cases for how $\gamma_1, \gamma_2$ can interact with each other.
\begin{itemize}
	\item Type 0: $\gamma_1$ and $\gamma_2$ are contained in the same torus boundary component. If $\langle \gamma_1, \gamma_2 \rangle = H_1(T_{i_j}; \mathbb{Z}/2\mathbb{Z})$ for $j \in \{1, 2\}$, we have a generating set of representatives $(\widehat{\gamma_1}, \widehat{\gamma_2}) = \pi_1(T_{i_j})$. By Lemmas \ref{lma:hyperbolicgoodgluing} and \ref{lma:cablespacegoodgluing}, for any lift $\widehat{\rho_i}$ to $SL_2(\mathbb{C})$, $\text{tr}(\widehat{\rho_i}(\widehat{\gamma_1})) = \text{tr}(\widehat{\rho_i}(\widehat{\gamma_2})) = -2$, and since $\widehat{\gamma_1}$ and $\widehat{\gamma_2}$ commute, it follows that for all $\widehat{\rho_i}$, $\text{tr}(\widehat{\rho_i}(\widehat{\gamma_1}\widehat{\gamma_2})) = 2$, which contradicts either Lemma \ref{lma:hyperbolicgoodgluing} or Lemma \ref{lma:cablespacegoodgluing}.
	\item Type 1: $\gamma_1$ and $\gamma_2$ are contained in $T_{i_1}$ and $T_{i_2}$, respectively. Then let $a_j \in \pi_1(\partial T_{i_j})$ be a representative of $\gamma_j$, and let $b_j \in \pi_1(\partial T_{i_j})$ be a representative of an element of $H_1(\partial T_{i_j}; \mathbb{Z}/2\mathbb{Z})$ not in the kernel of $i_*$. By Case 1 in the proof of Lemma \ref{lma:cablespacegoodgluing}, cable spaces with base orbifold $A^2(q)$ where $q$ is even fall into this type.
	\item Type 2: $\gamma_1$ is contained in $T_{i_1}$ and $\gamma_2 = (\gamma_2^1, \gamma_2^2)$ is spanning, with $\gamma_2^k \in H_1(T_{i_k}; \mathbb{Z}/2\mathbb{Z})$. Let $a \in \pi_1(\partial T_{i_1})$ be a representative of $\gamma_1$, and let $g^k \in \pi_1(T_{i_k})$ be a representative of $\gamma_2^k$. This means that for any lift $\widehat{\rho_i}$ of $\rho_i$, $\text{tr}(\widehat{\rho_i}(a)) = -2$ and $\text{tr}(\widehat{\rho_i}(g^1)) = \text{tr}(\widehat{\rho_i}(g^2))$. Let $b \in T_{i_1}$ be a representative of the $\mathbb{Z}/2\mathbb{Z}$-homology class that is not $\gamma_1$ or $\gamma_2^1$. Then for all lifts $\widehat{\rho_i}$, $\text{tr}(\widehat{\rho_i}(b)) = -\text{tr}(\widehat{\rho_i}(g^1)) = -\text{tr}(\widehat{\rho_i}(g^2))$, which is a contradiction of either Lemma \ref{lma:hyperbolicgoodgluing} or Lemma \ref{lma:cablespacegoodgluing}. If $\gamma_1 = \gamma_2^1$, then both $\gamma_1$ and $\gamma_2^1$ are -2 for all lifts, are we are in the situation of Type 1. So this case can be subsumed into Type 1.
	\item Type 3: $\gamma_1 = (\gamma_1^1, \gamma_2^2)$ and $\gamma_2 = (\gamma_2^1, \gamma_2^2)$ are both spanning, with $\gamma_j^k \in H_1(T_{i_k}; \mathbb{Z}/2\mathbb{Z})$. Let $g_j^k \in \pi_1(T_{i_k})$ be a representative of $\gamma_j^k$. Then for any lift $\widehat{\rho_i}$ of $\rho_i$, $\text{tr}(\widehat{\rho_i}(g_j^1) =  \text{tr}(\widehat{\rho_i}(g_j^2))$. By cases 2 and 3 in the proof of Lemma \ref{lma:cablespacegoodgluing}, cable spaces with base orbifold $A^2(q)$ where $q$ is odd fall into this type.
\end{itemize}
Since components of Type 0 or Type 2 cause contradictions, they do not exist. So we can label the components $M_i$ to be of Type 1 or 3, depending on how the map $i_*$ behaves. We now deal with the possibilities for the JSJ decomposition of $M$.
\begin{itemize}
	\item All the $M_i$ are of Type 1. If we have components $M_i$ with boundary torus $T_{i_1}$ and $M_{i+1}$ with boundary torus $T_{{i+1}_2}$, let $\varphi: \pi_1(T_{i_1}) \to \pi_1(T_{{i+1}_2})$ be the gluing isomorphism. Since both $M_i$ and $M_{i+1}$ are of Type 1, there exist generating sets $a_1, b_1 \in \pi_1(T_{i_1})$ and $a_2, b_2 \in \pi_1(T_{{i+1}_2})$ such that for all lifts $\widehat{\rho_i}$ and $\widehat{\rho_{i+1}}$, $\text{tr}(\widehat{\rho_i}(a_1)) = \text{tr}(\widehat{\rho_{i+1}}(a_2)) = -2$. If $\varphi(a_1)$ is $\mathbb{Z}/2\mathbb{Z}$-homologous to $a_2$, then we can pick $\widehat{\rho_i}, \widehat{\rho_{i+1}}$ so that $\text{tr}(\widehat{\rho_i}(b_1)) = \text{tr}({\rho_{i+1}}(\varphi(b_2))) = -2$, and the traces are equal on all of $\pi_1(T_{i_1})$. If $\varphi(a_1)$ is not $\mathbb{Z}/2\mathbb{Z}$-homologous to $a_2$, we can pick a lift $\widehat{\rho_{i+1}}$ so that $\text{tr}(\widehat{\rho_{i+1}}(\varphi(a_1))) = -2$, and a lift $\widehat{\rho_i}$ such that $\text{tr}(\widehat{\rho_i}(\varphi^{-1}(a_2))) = -2$. Note that in components of type 1, the choices we make on signs of traces on one boundary torus are independent of choices we make on the other. Therefore, we can make such choices for all JSJ tori, and so in this case the JSJ decomposition of $M$ is $SL_2(\mathbb{C})$-compatible. 
	\item All of the $M_i$ are of Type 3. Then we may choose lifts $\widehat{\rho_i}$ such that all peripheral elements have trace 2. This satisfies the  $SL_2(\mathbb{C})$-compatibility condition, and we are done in this case. 
	\item There are both $M_i$ of Type 1 and $M_i$ of Type 3. Due to the assumption that the JSJ complementary regions are arranged in a circle, any chain of Type 3 complementary components must end in Type 1 components on both sides. Notice that a chain of Type 3 complementary components has the effect of composing a gluing homeomorphism between two Type 1 components. So we are in the same situation as when all the $M_i$ are of Type 1, where the gluing homomorphisms go through chains of Type 3 components. The same proof as that case goes through. So in this case, $M$ is also $SL_2(\mathbb{C})$-compatible.
\end{itemize}
So in all possible cases, the JSJ decomposition of $M$ is $SL_2(\mathbb{C})$-compatible, as desired.
\end{proof}

\subsection{Final background results}

\begin{lemma}\label{lma:verticesedges}
Suppose we are in the situation in the hypothesis of Theorem \ref{thm:main}. Then the number of complementary regions of $\overline{M \setminus \bigcup_{i=1}^nT_i}$, the number of punctured JSJ tori $T_i$, the number of JSJ complementary regions of $M(0)$, and the number of JSJ tori of $M(0)$ are all equal to $n$. The JSJ complementary regions of $M(0)$ have two boundary tori. In particular, the JSJ complementary regions of $M(0)$ are arranged in a circle. 
\end{lemma}

\begin{proof}
The once-punctured tori $T_i$ are disjoint, and loops around the punctures $\partial T_i$ are parallel slopes on the boundary torus of $M$. Drawing a path around the meridian of the boundary torus, we see that the complementary regions connect in a circle, where each complementary region is bounded by two once-punctured tori on both sides. Notice that since $M$ is a knot complement in $S^3$, one cannot go from one region between two once-punctured tori to another without crossing a once-punctured torus. Thus, each complementary region is bounded by exactly two once-punctured tori, and the number of complementary regions is equal to the number of once-punctured tori. Note also that the number of both JSJ tori complementary regions of $M(0)$ are also equal to $n$, since the JSJ complementary regions are obtained from the complementary regions of $\overline{M \setminus \bigcup_{i=1}^n T_i}$ by gluing a solid cylinder. This also implies that each JSJ complementary region is bounded by exactly two tori.
\end{proof}

\begin{remark}
The purpose of this lemma is to establish that the JSJ decomposition of $M(0)$ satisfies the conditions of Theorem \ref{thm:compatible}, and will hence be $SL_2(\mathbb{C})$-compatible. In addition, the fact that the number of JSJ complementary regions is equal to the number of JSJ tori will prove useful in a dimension count featured in the proof of Theorem \ref{thm:main}. In fact, this particular structure of the JSJ decomposition of $M(0)$ will be crucial in establishing that each JSJ component has a two-dimensional character variety. Thus, we must assume that $M$ is a hyperbolic knot complement. 
\end{remark}

We will need the following result on JSJ decompositions for the proof of the theorem.

\begin{theorem}[\cite{Aschenbrenner20123ManifoldG}]\label{thm:JSJ}
Let $N$ be a compact irreducible orientable 3-manifold with JSJ decomposition $\{T_i\}$. Then:
\begin{enumerate}
	\item Each complementary region of $\overline{N \setminus \bigcup\{T_i\}}$ is atoroidal or Seifert-fibered. 
	\item If $T_i$ cobounds Seifert fibered connected components $N_j, N_k$ of the complementary region, then their regular fibers do not match.
	\item If $T^2 \times I$ is a complementary region, then $N$ is a torus bundle.
\end{enumerate}
\end{theorem}

We also slightly modify a result of Tillmann in \cite{Tillmann2003VarietiesAT} in order to cover non-connected essential surfaces.

\begin{lemma}[\cite{Tillmann2003VarietiesAT}]\label{lma:tillmann}
Let $S$ be an essential surface in an orientable irreducible 3-manifold $M$ with complementary regions $M_1, \dots, M_n$. Let $x_\infty$ be an ideal point on a curve inside the character variety $X(M)$, and the limiting character $\chi_\infty$ at $x_\infty$. There is a natural restriction map 
\begin{equation}
	r: X(M) \to X(M_1) \times \dots \times X(M_n)
\end{equation}
Then $S$ is detected by $x_\infty$ if and only if the limiting character $(\chi_1, \dots, \chi_n) \in X(M_1) \times \dots \times X(M_n)$ has the following properties:
\begin{enumerate}
	\item $\chi_i$ is finite for each $i = 1, \dots, n$, i.e. the values of $\chi_i$ on group elements have no poles at $x_\infty$ when restricted to $X(M_i)$.
	\item For any proper essential subsurface $S'$ of $S$, there exists a complementary region of $\overline{M \setminus S'}$ on which the limiting character is not finite.
	\item $\chi_i$ match via gluing homeomorphisms on the boundaries of $M_1, \dots, M_n$.
	\item $\chi_i$ restricted to any boundary component of $M_1, \dots, M_n$ is reducible.
	\item There is a connected open neighborhood $U$ of $x_\infty$ in $X(M)$ such that $r(U)$ contains an open neighborhood of $(\chi_1, \dots, \chi_n)$ but not $(\chi_1, \dots, \chi_n)$ itself.
\end{enumerate}
\end{lemma}

\begin{remark}
It was alluded to in \cite{Tillmann2003VarietiesAT} that Conditions 1, 3, 4, and 5 were sufficient if $S$ was connected, but not if $S$ has multiple components, since the condition S3 for surface detection would not be satisfied; it is also stated that Condition 2 rectifies this discrepancy. 
\end{remark}

%
%

\section{Proof of main theorem}\label{sec:proof}


For the convenience of the reader, we recall the theorem statement here.

\main*

\begin{proof}[Proof of Theorem \ref{thm:main}]
Suppose we are in the situation in the theorem statement. Let $[\beta] \in \pi_1(M)$ be the canonical longitude of the torus boundary, i.e. the boundary slope of the once-punctured tori. We use the following observations / notations, coming from Lemma \ref{lma:verticesedges}. 
\begin{enumerate}
	\item Let $\{H_i\}_{i=1}^n$ be the components of $\overline{M \setminus \bigcup_{i=1}^nT_i}$, and let $\{\mathcal{O}_i\}_{i=1}^n$ be the JSJ components of $M(0)$. Note that the number of complementary regions is equal to the number of JSJ tori, and both are equal to $n$.
	\item The JSJ component $\mathcal{O}_i$ is bounded by the tori $\widehat{T}_{i_1}, \widehat{T}_{i_2}$. 
	\item $T_{i_j} \subset \partial H_i$ are the incompressible punctured tori which cap off to $\widehat{T}_{i_j}$, and the commutators of $\pi_1(T_{i_1})$ and $\pi_1(T_{i_2})$ are equal in $\pi_1(M)$ (they are represented by the curve $\beta$)
	\item In the JSJ decomposition of $M(0)$, let $\widehat{T}_{i_1}$ be glued to $\widehat{T}_{(i-1)_2}$ and $\widehat{T}_{i_2}$ be glued to $\widehat{T}_{(i-1)_1}$. (Here $i \in \mathbb{Z}/n\mathbb{Z}$.) 
\end{enumerate} 
	
\begin{figure}[h]\label{fig:complementaryregionsK}
	\centering
	\includegraphics[scale=.25]{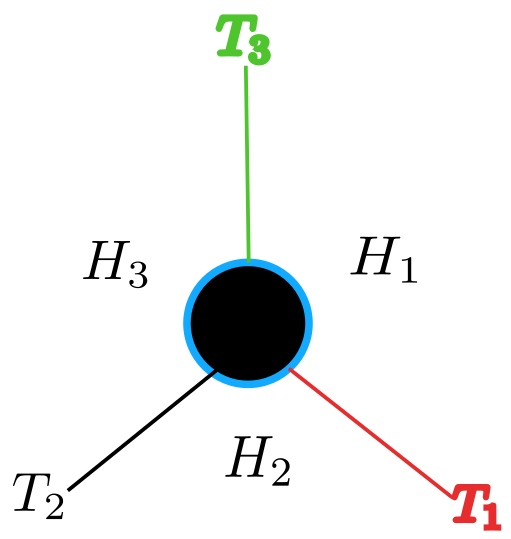}
	\caption{A schematic of how the once-punctured tori and complementary regions are arranged in the knot complement.}
\end{figure}
	
The picture is as follows: Each $\mathcal{O}_i$ is Seifert-fibered or hyperbolic. We also have that $\partial H_i = T_i \cup A_i \cup T_{i+1}$ is a closed genus 2 surface, which can be viewed as two once-punctured tori glued together via an annulus. The 3-manifold $\mathcal{O}_i$ is obtained by capping off $H_i$ with an annulus (representing a portion of the glued-in torus when perfoming 0-surgery on $M$).
	
\begin{figure}[h]
	\centering
	\includegraphics[scale=.3]{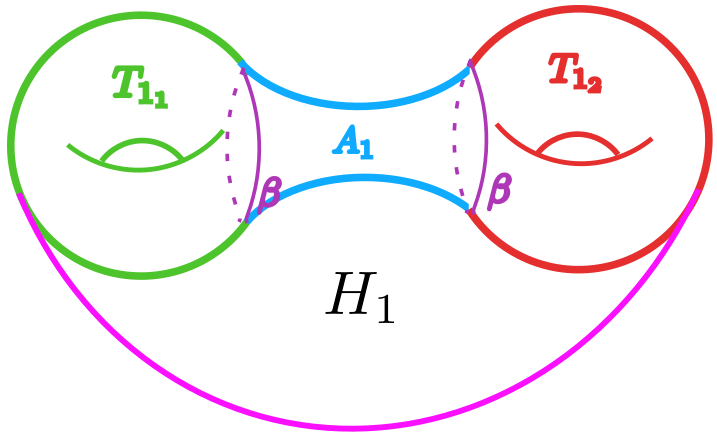}
	\caption{Rough drawing of what each complementary region looks like. Colors correspond to the previous figure.}
\end{figure}
	
If some $\mathcal{O}_i$ is Seifert-fibered over a 2-orbifold $S_i$, then $S_i$ must have two circle boundary components. By Theorem 13.3.6 in \cite{Thurston1979TheGA}, if $S_i$ is non-hyperbolic it must be the annulus, which is ruled out by assumption. So we can assume that all the JSJ components $\mathcal{O}_i$ are either hyperbolic or Seifert-fibered over a hyperbolic 2-orbifold. 
	
\medskip
	
We now examine the dimensions of $X(H_i)$ and $X(\mathcal{O}_i)$. Recall that $\partial H_i$ is a closed genus 2 surface. By Theorem 5.6 in \cite{Thurston1979TheGA}, $\dim_{\mathbb{C}}(X(H_i)) \geq -\frac{3}{2}\chi(\partial H_i) = 3$. Recall that $\mathcal{O}_i$ is a 3-manifold with two torus boundary components. By the same theorem, we know that $\dim_{\mathbb{C}}(X(\mathcal{O}_i)) \geq 2$. 
	
\medskip
	
We have the inclusions $\pi_1(H_i) \hookrightarrow \pi_1(M)$, inducing the restriction maps 
\begin{equation}
	r_i: X(M) \to X(H_i)
\end{equation}
Take the product of all these maps to get 
\begin{equation}
	r: X(M) \to X\left(\bigcup H_i\right) = \prod_{i=1}^nX(H_i) = \mathcal{X}
\end{equation}
Notice that $\mathcal{X}$ is an affine variety of dimension at least $3n$. We will now show that $\text{Im}(r)$ is a one-dimensional subvariety in $\mathcal{X}$. Let $V'$ be the algebraic set inside $\mathcal{X}$ defined by the following equations:
\begin{enumerate}
	\item Each $H_i$ has a coordinate associated with the mutual commutator of the generators of $\pi_1(T_{i_1})$ and $\pi_1(T_{i_2})$. This commutator lifts to $\beta \in \pi_1(M)$, so all the commutators must have the same trace in order to be in the image of $r$. Thus, there are $n - 1$ equations that equate the trace polynomials of these commutators.
	\item There are $n$ gluings between $T_{i_2}$ and $T_{(i+1)_1}$, with $i \in \mathbb{Z}/n\mathbb{Z}$. There are three equations per gluing here. If $\pi_1(T_{i_2}) = \langle a_2, b_2 \rangle$, $\pi_1(T_{(i+1)_1}) = \langle a_1, b_2 \rangle$, and the gluing map is $\varphi: \pi_1(T_{i_2}) \to \pi_1(T_{(i+1)_1})$, then the three equations are
	\begin{equation}
		\text{tr}(a_2) = \text{tr}(\varphi(a_2)) \ \ \ \ \ \text{tr}(b_2) = \text{tr}(\varphi(b_2)) \ \ \ \ \ \text{tr}(a_2b_2) = \text{tr}(\varphi(a_2b_2))
	\end{equation}
	However, only two of these equations are necessary. Let $x = \text{tr}(a_2)$ and $x' = \text{tr}(\varphi(a_2))$, with similar notation for $y, y'$ for $b_2$ and $z, z'$ for $a_2b_2$. Without loss of generality, suppose that the first two equations apply. Write them as $x = x'$ and $y = y'$. By the $n - 1$ equations that equate the commutators of all the trace functions, we also have that 
	\begin{equation}
		x^2 + y^2 + z^2 - xyz - 2 = x'^2 + y'^2 + z'^2 - x'y'z' - 2
	\end{equation}
	By the other two equations, we see that 
	\begin{equation}
		z^2 - xyz = z'^2 - xyz' \Longrightarrow (z^2 - z'^2) - (xyz - xyz') = (z - z')(z + z' - xy) = 0
	\end{equation}
	Thus, the algebraic set $V'$ includes $2n$ equations, which are of the form $x = x'$ and $y = y'$. 
\end{enumerate}
Let $V$ be the subvariety of $V'$ which takes the $z = z'$ components of the equations of the second type, as described above. Notice also that $\text{Im}(r) \subset V$, since any element in the image or $r$ must satisfy the gluing equations. Let $X^{irr}(M)$ be the Zariski open subset consisting of traces of irreducible representations in $X(M)$, and let $V^{irr}$ be the Zariski open subset of $V$ defined by $\text{tr}(\beta) \neq 2$. Note that $V^{irr} \subset \text{Im}(X^{irr}(M))$, since the traces on $X(H_i)$ restrict to matching irreducible traces on the once-punctured tori in their boundary, so the lifts of the conjugacy classes of irreducible representations on these once-punctured tori can be glued, giving an irreducible trace on $X(M)$. (See, for instance, Lemma 6 in \cite{Paoluzzi2010ConwaySA}.) Since $V$ is a component of an algebraic set defined by $3n - 1$ equations inside $\mathcal{X}$, $\dim(V) \geq 1$, and since by assumption $X^{irr}(M)$ is one-dimensional, $\text{dim}(V^{irr}) \leq 1$. If the dimension of $V^{irr}$ is zero, then $V^{irr} \subset V$ is a Zariski open subset such that $\dim(V^{irr}) < \dim(V)$, meaning that $V^{irr}$ is empty; this is a contradiction since by assumption, there exist representations with $\text{tr}(\beta) \neq 2$. So $\dim(V) = \dim(V^{irr}) = 1$. This also means that $\dim(\mathcal{X}) = 3n$, and so $\dim(X(H_i)) = 3$ for all $i$. The subvariety $X(\mathcal{O}_i) \subset X(H_i)$ is defined by setting the trace of $\beta \in \pi_1(\partial H_i)$ equal to 2, but by assumption, the trace of $\beta$ is nonconstant on $X(M)$, and hence nonconstant in $X(H_i)$ as well. Therefore, $\dim(\mathcal{O}_i) = 2$.
	
\medskip
	
By Lemma \ref{lma:twodim}, all $\mathcal{O}_i$ must be either cable spaces or hyperbolic. By Theorem \ref{thm:compatible}, $\{\widehat{T}_i\}_{i=1}^n$ is an $SL_2(\mathbb{C})$-compatible JSJ decomposition, so there exist representations $\rho_i: \pi_1(\mathcal{O}_i) \to SL_2(\mathbb{C})$ so that the cusp group elements have matching traces. In fact, by the proof of Theorem \ref{thm:compatible}, the $\rho_i$ can be chosen as lifts of holonomy representations of either $\mathcal{O}_i$ or their underlying 2-orbifolds $S_i = A^2(p_i)$. There are quotient maps $q_i: \pi_1(H_i) \to \pi_1(\mathcal{O}_i)$ that kill the loop corresponding to the punctures on the incompressible once-punctured tori embedded in the boundary of $H_i$; in other words, $q_i$ sends a representative of the core of $A_i$ to the identity. Let $\chi_i' = \text{tr}(\rho_i \circ q_i)$, and let $\chi_0 = (\chi_1', \dots, \chi_n')$. 

\medskip
	
Note that $V$ contains $\chi_0$, since we chose the coordinates of $\chi_0$ to be $SL_2(\mathbb{C})$-compatible characters, but $\chi_0 \notin V^{irr}$. If $\chi_0$ was isolated, $\{\chi_0\}$ would be a zero-dimensional component of $V$, which contradicts the fact that $V$ was defined with $3n - 1$ equations in a $3n$-dimensional variety. In addition, since the coordinate functions on $\chi_0$ are irreducible on each $X(H_i)$, $\chi_0$ cannot be surrounded by restrictions of reducible representations of $\pi_1(M)$. These observations combined with the fact that $V^{irr}$ is Zariski-open in $V$ shows that we may pick a sequence of points $\{\chi_j\}_{j=1}^\infty \subset V^{irr} \subset \text{Im}(r)$ approaching $\chi_0$. Let $\{\alpha_j\} \in X(M)$ be such that $r(\alpha_j) = \chi_j$. Suppose for contradiction that up to subsequence, $\{\alpha_j\}$ converges to a character $\alpha_\infty \in X(M)$. Then $\alpha_\infty$ must be the trace of some $\rho_\infty: \pi_1(M) \to SL_2\mathbb{C}$. Since $r(\alpha_\infty) = \chi_0$, it follows that for all $i$, $\rho_\infty|_{\pi_1(H_i)}$ is conjugate to $\rho_i \circ q_i$, since the trace of $\rho_i \circ q_i$ is $\chi_i'$. In particular, this means that for all $i \in \mathbb{Z}/n\mathbb{Z}$, $ \rho_i|_{\pi_1(\widehat{T}_{i_2})} \circ q_i|_{\pi_1(T_{i_2})}$ and $\rho_{i+1}|_{\pi_1(\widehat{T}_{(i+1)_1})} \circ q_i|_{\pi_1(T_{(i+1)_1})}$ are conjugate representations of the fundamental group of the once-punctured torus, which is the free group on two generators. We are left with the following three cases: $T_{i_2}$ and $T_{(i+1)_1}$ glue to form the incompressible once-punctured torus $T_k \subset M$ that caps off to $\widehat{T}_k \subset M(0)$ such that:
\begin{enumerate}
	\item $\widehat{T}_k$ bounds a Seifert-fibered and a Seifert-fibered space. Let $\pi_1(\widehat{T}_{i_1}) = \langle m_1, \ell_1 \rangle$ and $\pi_1(\widehat{T}_{i_2}) = \langle m_2, \ell_2 \rangle$ be the copies of $\widehat{T}_j$ in the two Seifert-fibered spaces with holonomies $\rho_1', \rho_2'$. Here $\ell_k$ is the regular fiber. Then up to conjugacy,
	\begin{equation}
		\rho_k'(m_k) = \pm\begin{pmatrix}1&1\\0&1\end{pmatrix}\ \ \ \ \ \pm\rho_k'(\ell_k) = \begin{pmatrix}1&0\\0&1\end{pmatrix}
	\end{equation} 
	for $k = 1, 2$, and some canonical choice of longitude $\ell_k$ and a choice of sign determined by $SL_2(\mathbb{C})$-compatibility. By Theorem \ref{thm:JSJ}, the fibers of the Seifert fibrations do not match, i.e. if $\varphi: \langle m_1, \ell_1 \rangle \to \langle m_2, \ell_2 \rangle$ is the gluing map, $\varphi(\ell_1) \neq \varphi(\ell_2)$. Thus, $\varphi^{-1}(\ell_2)$ is a nontrivial simple closed curve in $\pi_1(\widehat{T}_{i_1})$ that is not $\ell_1$; this means it must be of the form $m_1^p\ell_1^q$ with $p \neq 0$. However, 
	\begin{equation}
		\rho_1'(m_1^p\ell_1^q) = \pm \begin{pmatrix}1&p\\0&1\end{pmatrix} \neq \pm \begin{pmatrix}1&0\\0&1\end{pmatrix}
	\end{equation}
	which means that the two representations are not conjugate under the gluing map $\varphi$. This contradicts the earlier assertion. 
	\item $\widehat{T}_k$ bounds a Seifert-fibered and a hyperbolic space. The hyperbolic space will have a holonomy representation $\rho_h$ that restricts to the torus cusp $\langle m_h, \ell_h \rangle$ as follows:
	\begin{equation}
		\rho_h(m) = \pm\begin{pmatrix}1&1\\0&1\end{pmatrix} \ \ \ \ \ \rho_h(\ell_h) = \pm\begin{pmatrix}1&\tau\\0&1\end{pmatrix}
	\end{equation}
	where $\text{Im}(\tau) > 0$, and the choice of sign is determined by $SL_2(\mathbb{C})$-compatibility. We call $\tau$ the \emph{cusp shape} of the representation $\rho_h$. As seen previously, the cusp shape for a Seifert-fibered manifold is 0, but it is non-zero for a hyperbolic manifold, so this leads to a contradiction.
	\item $\widehat{T}_k$ bounds a hyperbolic and a hyperbolic space. This is excluded by hypothesis. 
\end{enumerate}
In all cases, we get a contradiction. Thus, a subsequence of $\{\alpha_j\}$ converges to an ideal point $\alpha_\infty$ of $X(M)$. By construction, $\lim_{j\to\infty}\text{tr}(\alpha_j)|_{\pi_1(H_i)} = \chi_i'$ is the limiting character of the ideal point, and the component whose projectivization contains $\alpha_\infty$ must contain irreducible representations. By the assumptions, it follows that $\alpha_\infty$ is an ideal point on a norm curve $C$. Since the limiting characters $(\chi_1', \dots, \chi_n')$ satisfy the conditions in Lemma \ref{lma:tillmann}, $\alpha_\infty \in \widetilde{X(M)}$ detects $\bigcup T_i$.  Since the limiting character on each $H_i$ is given by the trace of $\rho_i \circ q: \pi_1(H_i) \to SL_2(\mathbb{C})$, which is irreducible on all components, by Proposition \ref{prop:limitingrep} $A_{k(C)}$ tautologically extends over $\alpha_\infty$, and so we are done. 
\end{proof}

\section{Consequences, examples, and conjectures}\label{sec:conjectures}


\subsection{Two-bridge knots}

\begin{question}\label{conj:jsjtori}
Let $M$ be a one-cusped hyperbolic 3-manifold with a family of $n$ mutually disjoint incompressible $p_i$-punctured JSJ tori $\{T_i\}_{i = 1}^n$ with slope $\beta$ such that the complementary components of $M(\beta)$ are not $T \times I$. When is $\bigcup_{i=1}^n\{T_i\}$ detected by an ideal point of the canonical component $X(M)$, and if so, when does $A_{k(C)}$ tautologically extend over that ideal point?
\end{question}

Question \ref{conj:jsjtori} is one generalization of Theorem \ref{thm:main}. From this perspective, we now analyze punctured JSJ tori in two-bridge knots. The JSJ structure of toroidal Dehn fillings of two-bridge knots are determined in Theorem 3.10 of \cite{Patton1995IncompressiblePT} and and Proposition 4.1 of \cite{Sekino}. We also include the cases of the twice-punctured tori in the figure-eight knot, which was stated in Section 7 of \cite{Tillmann2005DegenerationsOI}. By Corollary 1.2 in \cite{IchiharaMasai}, this is complete account of JSJ decompositions of toroidal Dehn fillings of two-bridge knots. 

\begin{theorem}[\cite{Patton1995IncompressiblePT}, \cite{Tillmann2005DegenerationsOI}, \cite{Sekino}]\label{thm:patton}
Let $M = S^3 \setminus K_{p/q}$ be a two-bridge knot complement associated to $p/q \in \mathbb{Q}$, and 
\begin{equation}
	p/q = r + \frac{1}{b_1 - \frac{1}{b_2}} \ \ \ \ \ r, b_1, b_2 \in \mathbb{Z}
\end{equation}
Then:
\begin{enumerate}
	\item If $|b_1|, |b_2| > 2$ are even then $M$ admits two once-punctured JSJ tori with slope 0, and the JSJ complementary regions of $M(0)$ are cable spaces.
	\item If $|b_1| = 2$ and $|b_2| > 2$ is even, then $M$ admits one once-punctured JSJ torus with slope 0, and the JSJ complementary region of $M(0)$ is a cable space.
	\item If $b_1 = -2, b_2 = 2$ then $M$ is the figure-eight knot complement, with one once-punctured JSJ torus with slope 0 such that the JSJ complementary region of $M(0)$ is $T^2 \times I$.
	\item The figure-eight knot complement also contains two twice-punctured JSJ tori with slopes $\pm 4$ such that the JSJ complementary region of $M(\pm 4)$ are the twisted $I$-bundle of the Klein bottle and the trefoil knot complement.
	\item If $b_1$ is odd and $|b_2| > 2$ is even, then $M$ admits two twice-punctured JSJ tori with slope $2b_2$ such that the JSJ complementary regions of $M(2b_2)$ are the twisted $I$-bundle of the Klein bottle, a $(2, 2m+1)$ torus knot complement for some $m \neq 0, -1$, and a cable space.
	\item If $b_1$ is odd and $|b_2| = 2$, then $M$ is the $m$-twist knot complement. In this case, $M$ admits one once-punctured JSJ torus with slope 0, and the JSJ complementary region of $M(0)$ is a cable space. 
	\item The $m$-twist knot complement also contains one twice-punctured JSJ torus with slope 4. The JSJ complementary regions of $M(4)$ are the twisted $I$-bundle of the Klein bottle and a $(2, 2m+1)$
\end{enumerate}
\end{theorem}

We will address the status of tautological extension of $A_{k(C)}$ in all of these cases, through either proof or conjecture. 

\begin{corollary}
The families of Seifert surfaces in cases 1, 2, and 6 of Theorem \ref{thm:patton} are detected by an ideal point of a norm curve $C$ of the character variety. In addition, $A_{k(C)}$ tautologically extends over this ideal point.
\end{corollary}

\begin{proof}
Theorem 1.2 in \cite{Macasieb2009OnCV} combined with Theorem 7.3 in \cite{Petersen2014GonalityAG} implies that all components of the character varieties of the two-bridge knots in Theorem \ref{thm:patton} containing traces of irreducible components are norm curves. Then Cases 1, 2, and 6 of Theorem \ref{thm:patton} combined with Theorem \ref{thm:main} implies the conclusion. 
\end{proof}

For an explicit example, we analyze the punctured JSJ tori in the figure-eight knot, which encompass cases 3 and 4 in Theorem \ref{thm:patton}. Since the figure-eight knot is fibered, it is known that the Seifert surface in case 3 is not detected by an ideal point of a norm curve. This is a case excluded in Theorem \ref{thm:main}, since the JSJ complementary region of $M(0)$ is $T^2 \times I$. 

\medskip

One can see using the $A$-polynomial \cite{Cooper1994PlaneCA} that the twice-punctured JSJ tori in case 4 of Theorem \ref{thm:main} are each detected by an ideal point of the canonical component of the figure-eight knot. The behavior of the limiting character in this case was computed in \cite{Tillmann2005DegenerationsOI}. Let $M$ be the figure-eight knot complement.  Let $S \subset M$ be the essential twice-punctured torus with slope $(\pm4, 1)$. 

\newpage

\begin{figure}[h]
	\centering
	\includegraphics[scale=.5]{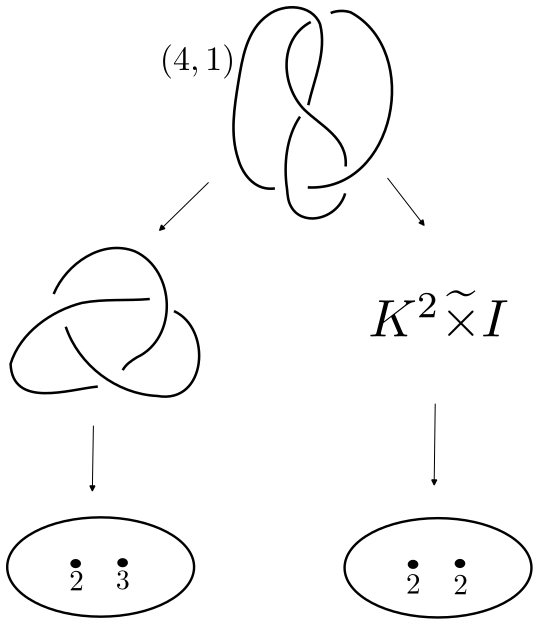}
	\caption{The splitting of $(4, 1)$ surgery on the figure-eight knot, where the Seifert fiberings are shown. The limiting character reflects this information.}
\end{figure}

The fundamental group of the trefoil knot complement $M' \subset M(\pm 4, 1)$ is given by 
\begin{equation}
	\pi_1(M') = \langle u, v \mid u^3 = v^2 \rangle
\end{equation}
The boundary torus of $M'$ is generated by the meridian $u^{-1}v$ and longitude $u^3$. Note that $M'$ is Seifert fibered over the hyperbolic 2-orbifold $D^2(2, 3) = S^2(2, 3, \infty)$.

\medskip

The fundamental group of $K \tilde{\times} I \subset M(\pm 4, 1)$, the nontrivial $I$-bundle over the Klein bottle, is given by 
\begin{equation}
	\pi_1(K \tilde{\times} I) = \langle s, t \mid s^2t^2 = 1 \rangle
\end{equation}
The boundary torus is generated by the meridian $st$ and the longitude $s^2$. We note that $K \tilde{\times} I$ fibers over a Euclidean 2-orbifold $D(2, 2)$, and $K \tilde{\times} I$ itself has a Euclidean structure. So there is no hyperbolic holonomy representation into $PSL_2(\mathbb{C})$ to speak of, as we will observe in the behavior of the limiting character on $\pi_1(K \tilde{\times} I) \subset \pi_1(M(4, 1))$. 

\medskip

The representation of $\pi_1(M')$ corresponding to the holonomy of $D(2, 3)$ is given by 
\begin{equation}
	\rho_1(u) = \begin{pmatrix}1&-1\\1&0\end{pmatrix} \ \ \ \ \ \rho_1(v) = \begin{pmatrix}0&-1\\1&0\end{pmatrix}
\end{equation}
On the boundary torus, this restricts to 
\begin{equation}
	\rho_1(u^{-1}v) = \begin{pmatrix}1&0\\1&1\end{pmatrix} \ \ \ \ \ \rho_1(u^3) =  \begin{pmatrix}-1&0\\0&-1\end{pmatrix}
\end{equation}
In \cite{Tillmann2005DegenerationsOI}, it is computed that the limiting character restricts to the trace of $\rho_1$ on $\pi_1(M')$. Meanwhile, on $K \tilde{\times} I$, the limiting character approaches the trace of 
\begin{equation}
	\rho_\infty(s) = \begin{pmatrix}0&1\\-1&0\end{pmatrix} \ \ \ \ \ \rho_\infty(t) = \begin{pmatrix}0&-1\\1&0\end{pmatrix}
\end{equation}
So on the boundary torus, the limiting character tends to the trace of
\begin{equation}
	\rho_\infty(st) = \begin{pmatrix}1&0\\0&1\end{pmatrix} \ \ \ \ \ \rho_\infty(s^2) = \begin{pmatrix}-1&0\\0&-1\end{pmatrix}
\end{equation}
The fact that $\chi_\infty|_{\pi_1(M')}$ is the trace of an irreducible holonomy representation gives us

\begin{corollary}\label{cor:figureeight}
$A_{k(C)}$ tautologically extends over both ideal points of $X(M)$. 
\end{corollary}

\begin{remark}
The tautological extension comes from the limiting character on the trefoil knot complement, which is a JSJ complementary region in $M(\pm 4, 1)$. In particular, the extension comes from the holonomy structure on the underlying 2-orbifold of the Seifert fibration on the trefoil, as is often the case in applications of Theorem \ref{thm:main}. However, the limiting character on $\pi_1(K \tilde{\times} I)$ is the character of a reducible representation, and does not contribute to the tautological extension. This does not occur in the situation of Theorem \ref{thm:main}. It would be interesting to study the behavior of limiting characters on submanifolds that are Seifert-fibered over non-hyperbolic 2-orbifolds, such as the twisted $I$-bundle over the Klein bottle. It is known that many JSJ decompositions on exceptional surgeries (e.g. twist knots) feature this space as a complementary component. This study will be partially carried out in \cite{newpaper}. 
\end{remark}

The remaining cases of Theorem \ref{thm:patton} are 5 and 7. Since the torus knot complement involved in the JSJ decomposition of the toroidal Dehn fillings capping off the twice-pnctured tori is Seifert-fibered over a hyperbolic 2-orbifold, we find the following conjecture to be reasonable.

\begin{conjecture}
The twice-punctured tori in cases 5 and 6 of Theorem \ref{thm:patton} are detected by an ideal point of a norm curve in the character variety, and $A_{k(C)}$ tautologically extends over this ideal point.
\end{conjecture}

This conjecture, which generalizes Corollary \ref{cor:figureeight}, will be proven in upcoming work \cite{newpaper}. 

\subsection{The roots of unity phenomenon}

Recall the ``roots of unity phenomenon" first proven in \cite{Cooper1994PlaneCA} and elaborated upon in \cite{Dunfield2008AP1}:

\begin{theorem}[\cite{Cooper1994PlaneCA}]
	Suppose $S$ is an essential surface in $M$ with nonempty boundary detected by the ideal point $x$. Let $b = [\partial S] \in \pi_1(\partial M)$ be the fundamental group element associated to the boundary of $S$, which is an element of the peripheral torus. Then the limiting character at $b$ approaches $\lambda + \lambda^{-1}$ at $x$, and all other peripheral elements approach infinity. Here $\lambda = \xi_n$, a primitive $n$th root of unity. In addition, $n$ divides the number of boundary components of any connected component of $S$.
\end{theorem}

\begin{conjecture}\label{conj:main}
Let $M$ be a one-cusped hyperbolic 3-manifold, and let $x \in \widetilde{C}$ be an ideal point detecting the essential surface $S \subset M$ with limiting eigenvalue a primitive $n$th root of unity and boundary slope $\beta$. Let 
\begin{equation}
	\beta' = \begin{cases}n\beta & n \text{ odd} \\ \frac{n}{2}\beta & n \text{ even} \end{cases}
\end{equation} 
Let $S(\beta')$ be the surface $S$ but with cone points of order $n$ or $\frac{n}{2}$ (as prescribed above) instead of punctures. Suppose that $S(\beta')$ is a Euclidean orbifold, and there exists a connected component $M' \subset M(\beta') \setminus S(\beta')$ that is either hyperbolic or Seifert-fibered over a hyperbolic 3-manifold. Then the limiting character at $x$ restricted to $M'$ is equal to the trace of the holonomy representation of $M'$, and $A_{k(C)}$ tautologically extends over $x$.
\end{conjecture}

This conjecture is inspired by the roots of unity phenomenon, which implies that the limiting character at the ideal point descends to a character on $M'$. Stemming from the discussion in this paper, the natural candidate for the limiting character is the trace of the irreducible holonomy representation on $M'$, which would provide a basis for tautological extension over that ideal point.

\medskip

We now discuss one such situation where Conjecture \ref{conj:main} holds, with $n = 4$ and $M$ is a knot complement in $S^3$. In \cite{Paoluzzi2010ConwaySA}, it is shown that ideal points exist that detect systems of Conway spheres. We summarize these results here:

\begin{theorem}[\cite{Paoluzzi2010ConwaySA}]
Let $K$ be a hyperbolic knot complement in $S^3$ with a Bonahon-Siebenmann system of essential Conway spheres $\{C_i\}$. Let $\{M_j\}$ be the components of the complement of $\{C_i\}$ in $K$. Then the following hold:
\begin{enumerate}
	\item $\bigcup_{i} C_i$ is detected by an ideal point of $X(K)$. Let $\chi_\infty$ be the limiting character at $x$.
	\item Let $m \in \pi_1(K)$ be a meridian of $K$. Then $\chi_\infty = 0$, and in fact, the representation the limiting character comes from is conjugate to
	\begin{equation}
		\rho_\infty(m) = \pm\begin{pmatrix}i & * \\ 0 & -i\end{pmatrix}
	\end{equation}
	\item Each $\mathcal{O}_i$ is either hyperbolic or Seifert fibered over a hyperbolic 2-orbifold.
	\item $\chi_\infty$ corresponds to the trace of the hyperbolic holonomy representation of $\pi_1(\mathcal{O}_i)$, where $\mathcal{O}_i$ is an orbifold obtained from $M_i$ by endowing each knot component with ramification index 2. Note that $\mathcal{O}_i$ has cusps homeomorphic to the pillowcase orbifold.
\end{enumerate}
\end{theorem}

Combining the above results with Proposition \ref{prop:limitingrep} yields

\begin{corollary}
Let $x$ be an ideal point of $X(K)$ that detects a system of Conway spheres. If $x$ lies on a norm curve $C$ of $X(K)$, then $A_{k(C)}$ tautologically extends over $x$.
\end{corollary}

\begin{remark}
The proof of Theorem \ref{thm:main} replicates the approach of \cite{Paoluzzi2010ConwaySA}. The ideas in \cite{Paoluzzi2010ConwaySA} deal with Bonahon-Siebenmann decompositions using pillowcase orbifolds resulting from $(2, 0)$-Dehn fillings on Conway spheres. In a parallel situation to that of \cite{Paoluzzi2010ConwaySA}, Theorem \ref{thm:main} deals with JSJ decompositions with essential tori resulting from Dehn filling on punctured tori, and the limiting eigenvalue was 1.
\end{remark}

In \cite{Dunfield2008AP1}, the author found examples of non-trivial roots of unity at ideal points, such as $m137$, which has limiting eigenvalue a sixth root of unity. Tillmann \cite{Tillmann2003VarietiesAT} also computed that the limiting character at this ideal point is irreducible. In light of these examples, we make the following conjecture, which is the $n = 6$ case of Conjecture \ref{conj:main}.

\begin{conjecture}
Let $M$ be a one-cusped hyperbolic 3-manifold such that $X(M)$ satisfies hypotheses of Theorem \ref{thm:main}, and let $\{P_i\}_{i=1}^n$ be a family of non-isotopic thrice-punctured spheres that cap off to a Bonahon-Siebenmann decomposition of $M(3, 0)$ consisting of $(3, 3, 3)$ turnovers. Then $\bigcup_{i=1}^nP_i$ is detected by an ideal point of a norm curve, and $A_{k(C)}$ extends over that ideal point.  
\end{conjecture}

The manifold $m137$ is also notable because its canonical component contains no traces of reducible representations. The same is also true for knots with Alexander polynomial 1. In light of this, and the fact that $A_{k(C)}$ cannot tautologically extend over points corresponding to reducible traces, we conclude with the following question.

\begin{question}
Do there exist $M$ such that $A_{k(C)}$ extends tautologically over all of $X(M)$?
\end{question}

\subsection{Pretzel knots}

Punctured JSJ tori in alternating hyperbolic knots were classified in \cite{IchiharaMasai}. 

\begin{theorem}\cite{IchiharaMasai}
Let $M = S^3 \setminus K$ be a hyperbolic alternating knot complement. Suppose $M(r)$ is toroidal but not Seifert-fibered. Then $K$ is equivalent to one of:
\begin{enumerate}
	\item the figure-eight knot, with $r = 0, \pm 4$
	\item a two-bridge knot $K_{(b_1, b_2)}$ with $|b_1|, |b_2| > 2$, with $r = 0$ if $b_1, b_2$ are even and $r = 2b_2$ if $b_1$ is odd and $b_2$ is even
	\item a twist knot $K_{(2n, \pm 2)}$ with $|n| > 1$ and $r = 0, \pm 4$
	\item a pretzel knot $P(q_1, q_2, q_3)$ with $q_i \neq 0, \pm 1$ for $i = 1, 2, 3$ and $r = 0$ if $q_i$ are all odd, and $r = 2(q_2 + q_3)$ if $q_1$ is even and $q_2, q_3$ are odd
\end{enumerate}
\end{theorem}

Note that all possible punctured JSJ tori in the case of two-bridge knots have been addressed earlier in this section. In light of this, the next immediate generalization of Theorem \ref{thm:main} is for once and twice-punctured tori in pretzel knots. The following example should be highlighted in particular:

\begin{prop}[\cite{BaldwinSivek}]
Let $K_n$ be the $(-3, 3, 2n+1)$-pretzel knot, and $M_n$ its complement in $S^3$. Then $M_n$ admits a genus one Seifert surface which caps off to a JSJ torus in $M_n(0)$. The  complementary region of this JSJ torus in $M_n(0)$ is homeomorphic to the torus link $T_{2, 4}$, which is Seifert-fibered over the annulus.
\end{prop}

This example violates the second hypothesis of Theorem \ref{thm:main}. It would be interesting to determine whether or not the Seifert surface of $M_n$ is detected by an ideal point of a norm curve of $X(M_n)$. If so, then the limiting character at that ideal point must be reducible, providing the first example of an ideal point over which $A_{k(C)}$ does not tautologically extend. If not, this would be a rare example of a non-fiber Seifert surface not detected by an ideal point of a norm curve. Either way, Seifert surfaces of $M_n$ are expected to be examples of rare phenomena in Culler-Shalen theory. 

\bibliography{azumaya2.bib}
\bibliographystyle{plain}

\end{document}